\documentclass{amsart}[12pt]
\usepackage{amsmath, amsthm, amscd, amsfonts,}
\usepackage{graphicx,float}

\usepackage{amssymb}
\usepackage{pb-diagram}
\usepackage{lamsarrow}
\usepackage{pb-lams}

\newcommand{\PP}{{\mathbb{P}}}

\newcommand{\ga}{\alpha}

\newcommand{\gk}{\kappa}

\newcommand{\gt}{\tau}

\newcommand{\func }{\mathord{:}}

\newcommand{\restricted}{\mathord{\restriction}}

\newcommand{\ordered}[1]{\ensuremath{\langle #1 \rangle}}

\newcommand{\ordof}[2]{\ensuremath{\ordered{ #1 \mid #2 }}}

\DeclareMathOperator{\len}{l}

\DeclareMathOperator{\crit}{crit}
\DeclareMathOperator{\dom}{dom}

\DeclareMathOperator{\Col}{Col}
\DeclareMathOperator{\Add}{Add}
\DeclareMathOperator{\C}{C}
\DeclareMathOperator{\supp}{supp}
\DeclareMathOperator{\Ult}{Ult}
\DeclareMathOperator{\limdir}{lim\ dir}

\DeclareMathOperator{\mc}{mc}

\newcommand{\Es}{{\ensuremath{\bar{E}}\/}}

\newcommand{\VS}{V^*}
\newcommand{\MS}{{M^*}}

\newcommand{\NS}{{N^*}}

\newcommand{\MSt}{{M^*_\gt}}
\newcommand{\Mt}{{M_\gt}}

\newcommand{\ME}{{M_{\Es}}}

\newcommand{\MSE}{{M^*_{\Es}}}

\def\MQB{{\mathbb{Q}}}
\def\MRB{{\mathbb{R}}}

\def\k{\kappa}

\def\l{\lambda}

\newtheorem{theorem}{Theorem}[section]
\newtheorem{lemma}[theorem]{Lemma}

\newtheorem{corollary}[theorem]{Corollary}
\newtheorem{definition}[theorem]{Definition}

\newtheorem{remark}[theorem]{Remark}
\newtheorem{claim}[theorem]{Claim}
\newtheorem{question}[theorem]{Question}
\numberwithin{equation}{section}

\usepackage{pb-diagram}

\def\l{\lambda}

\def\rmark{\mbox{$\rm\bf\rule{0.06em}{1.45ex}\kern-0.05em R$}}
\def\pmark{\mbox{$\rm\bf\rule{0.06em}{1.45ex}\kern-0.05em P$}}
\def\nmark{\mbox{$\rm\bf\rule{0.06em}{1.45ex}\kern-0.05em N$}}
\def\vdash{\mbox{$\rm\| \kern-0.13em -$}}

\def\l{\lambda}

\def\rmark{\mbox{$\rm\bf\rule{0.06em}{1.45ex}\kern-0.05em R$}}
\def\pmark{\mbox{$\rm\bf\rule{0.06em}{1.45ex}\kern-0.05em P$}}
\def\nmark{\mbox{$\rm\bf\rule{0.06em}{1.45ex}\kern-0.05em N$}}
\def\vdash{\mbox{$\rm\| \kern-0.13em -$}}
\newcommand{\lusim}[1]{\smash{\underset{\raisebox{1.2pt}[0cm][0cm]{$\sim$}}
{{#1}}}}
\begin{document}

\title[$\text{HOD}$,~ $\text{V}$ and the $\text{GCH}$ ]{$\text{HOD}$,~ $\text{V}$ and the $\text{GCH}$}

\author[ M. Golshani.]{Mohammad Golshani}

\thanks{The author's research has been supported by a grant from IPM (No. 91030417).}
\maketitle

\begin{abstract}
Starting from large cardinals we construct a model of $ZFC$ in which the $GCH$ fails
everywhere, but such that $GCH$ holds in its $HOD$.  The result answers a question of Sy Friedman. Also, relative to the existence of
large cardinals, we produce a model of $ZFC+GCH$ such that $GCH$ fails everywhere in its $HOD$.
\end{abstract}
\maketitle

\section{Introduction}
In personal communication \cite{friedman}, Sy Friedman asked the author if we can have a model of $ZFC$ in which $GCH$ fails everywhere, but its $HOD$ satisfies the $GCH$. We
give an affirmative answer to his question by showing that the model of \cite{friedman-golshani} satisfies the required properties. To be more precise,
we prove the following

\begin{theorem}
Assume $V \models$``$ZFC+GCH+$there exists a $(\kappa+4)-$strong cardinal $\kappa$''. Then there is a generic extension $W$ of $V$ such that:

$(1)$ $\kappa$ remains inaccessible in $W$,

$(2)$ $V_\kappa^W=W_\kappa \models$``$~ZFC+ \forall \lambda, 2^{\lambda} = \lambda^{+3}$'',

$(3)$ $HOD^{W_\kappa} \models$``$GCH$''.

\end{theorem}

\begin{remark}
In fact it suffices to have a Mitchell increasing $\kappa^+$-sequence of extenders, each of which is $(\kappa+3)$-strong. Thus for the conclusion of Theorem 1.1 it suffices to have a cardinal $\kappa$ such that $o(\kappa) = \kappa^{+3} + \kappa^+$.
\end{remark}
The model $W$ we consider for the proof of Theorem 1.1 is the model $V^{\mathbb{P}_{\bar{E}}\times Col(\omega, \lambda)}$ of \cite{friedman-golshani}, but to show that its $HOD$ satisfies
the $GCH$ we need an analysis completely  different from \cite{friedman-golshani}.

Let us describe the main differences of the proof of Theorem 1.1 with that of \cite{friedman-golshani}. In \cite{friedman-golshani},  two forcing notions $\mathbb{P}_{\bar{E}}$
and $\mathbb{R}_{\bar{E}_\kappa}$ were defined such that:
\begin{enumerate}
\item $V_2=V_\kappa^{\mathbb{P}_{\bar{E}}\times Col(\omega, \lambda)} \models$``$~ \forall \eta, 2^{\eta} = \eta^{+3}$'', where
$\lambda$ is the minimal element of the Radin club added by $\mathbb{P}_{\bar{E}},$
\item $V_1=V_\kappa^{\mathbb{R}_{\bar{E}_\kappa}\times Col(\omega, \lambda)} \models$``$GCH$'',
\item $V_1 \subseteq V_2,$ as proved by finding a projection $\pi: \mathbb{P}_{\bar{E}} \rightarrow \mathbb{R}_{\bar{E}_\kappa},$
\item $V_1$ and  $V_2$ have the same cofinalities.
\end{enumerate}
As the guiding generics are coming from some homogeneous forcing notions (i.e., Cohen forcings and collapse forcings),  $HOD^{V_2}$  and
$V_1$ have different cardinal structure, and even if we can show that  $HOD^{V_2} \subseteq V_1$, there is no guarantee that $GCH$ holds
in $HOD^{V_2}$.

So, to prove the theorem, we build an inner model $N$ of $V^{\mathbb{P}_{\bar{E}}\times Col(\omega, \lambda)}$ which is different from $V^{\mathbb{R}_{\bar{E}_\kappa}\times Col(\omega, \lambda)}$. In fact we will show that there exists
 a cardinal
and $GCH$ preserving
generic extension $N$ of $V$ such that $HOD^{V^{\mathbb{P}_{\bar{E}}\times Col(\omega, \lambda)}} \subseteq N,$ and using it we conclude the result.

On the other hand, by a result of Roguski \cite{roguski}, any model $V$ of $ZFC$ has a class generic extension $V[G]$ such that $V$ is equal to the $HOD$ of
$V[G].$ Some generalizations of Roguski result are obtained by Fuchs-Hamkins and Reitz \cite{f-h-r}. These results can be used to find a model $W$
of $ZFC$ such that $GCH$ fails everywhere in its $HOD$. However in the constructions of \cite{roguski} and \cite{f-h-r}, the class generic extension
$V[G]$ fails to satisfy $GCH$. We modify the above constructions so that our final model satisfies the $GCH$, and use it to prove the following
theorem.
\begin{theorem}
Assume  $V\models$``$ZFC+GCH+ \kappa$ is a $(\kappa+4)$-strong cardinal. Then there exists a generic extension $W$ of $V$ such that

$(1)$ $\kappa$ remains inaccessible in $W$,

$(2)$ $V_\kappa^W=W_\kappa \models$``$~ZFC+GCH$'',

$(3)$ $HOD^{W_\kappa} \models$``$\forall \lambda, 2^{\lambda} = \lambda^{+3}$''.

\end{theorem}

In section 2 we present some preliminaries about projection between forcing notions and produce a new kind of projection which plays an essential role
in later sections of the paper. In section 3 we give the proof of Theorem 1.1 and in section 4 we complete the proof of Theorem 1.3. In the last section 5
we discuss some possible generalizations.

\section{Prikry type projections}
In this section, we present some definitions and results which appear in the following sections.
Let's start with the definition of a projection map between forcing notions.
\begin{definition}
Let $\PP, \MQB$ be two forcing notions. $\pi$ is a  projection from $\PP$ into $\MQB$ if $\pi: \PP \rightarrow \MQB,$ and it satisfies the following conditions:

$(1)$ $\pi(1_\PP)=1_{\MQB},$

$(2)$ $\pi$ is order preserving; i.e., $p \leq_\PP q \Rightarrow \pi(p) \leq_\MQB \pi(q),$

$(3)$ If $p\in \PP, q\in \MQB$ and $q \leq_\MQB \pi(p)$, then there exists $p^* \leq_\PP p$ such that $\pi(p^*) \leq_\MQB q.$
\end{definition}
If $\pi: \PP \rightarrow \MQB$ is a projection, then clearly $\pi[\PP]$ is dense in $\MQB.$ The next lemma shows that if $\PP$  projects into $\MQB,$ then a generic filter for $\PP$ yields  a generic filter for $\MQB.$
%and a generic for $\PP$ gives rise to a generic for $\MQB.$
%In this paper, we usually work with maps between forcing notions which satisfy a weaker property, called weak projection. The concept of weak projection was introduced by %Foreman and Woodin \cite{foreman-woodin} in their work for building a model of set theory in which the $GCH$
% fails everywhere.
%\begin{definition}
%Let $\PP, \MQB$ be two forcing notions. $\pi$ is a weak projection from $\PP$ into $\MQB$ if $\pi: \PP \rightarrow \MQB,$ and it satisfies the following conditions:

%$(1)$ $\pi(1_\PP)=1_{\MQB},$

%$(2)$ $\pi$ is order preserving; i.e., $p \leq_\PP q \Rightarrow \pi(p) \leq_\MQB \pi(q),$

%$(3)$ For any $p\in \PP,$ there exists $p' \leq_{\PP} p$ such that if $ q\in \MQB$ and $q \leq_\MQB \pi(p')$, then there exists $p^* \leq_\PP p'$ such that $\pi(p^*) %\leq_\MQB q.$
%\end{definition}
%It is clear that if $\pi: \PP \rightarrow \MQB$ is a weak projection from $\PP$ into $\MQB$ such that $1_{\PP}'=1_{\PP}$, then $\pi[\PP]$ is dense in $\MQB.$
\begin{lemma}
Let $\pi: \PP \rightarrow \MQB$ be a  projection from $\PP$ into $\MQB$,  let $\text{G}$ be $\PP$-generic over $\text{V}$, and let $\text{H} \subseteq \MQB$ be the filter
generated by $\pi[\text{G}].$ Then $\text{H}$ is $\MQB$-generic over $\text{V}$ and $\text{V[H]} \subseteq \text{V[G]}.$
\end{lemma}
Prikry type forcing notions arise in our work.
\begin{definition}
$\langle \PP, \leq, \leq^* \rangle$ is of Prikry type, iff

$(1)$ $\leq^* \subseteq \leq,$

$(2)$ For any $p\in \PP$ and any statement $\sigma$ in the forcing language $\langle \PP, \leq \rangle$, there exists $q \leq^* p$

$\hspace{0.5cm}$ which decides $\sigma.$
\end{definition}
The relation $\leq^*$ is usually called the Prikry relation.
The following is well-known.
\begin{lemma}
Assume $\langle \PP, \leq, \leq^* \rangle$ is of Prikry type, and suppose $\langle \PP, \leq^* \rangle$ is $\k$-closed, where $\k$ is regular uncountable. Then Forcing with $\langle \PP, \leq \rangle$ does not add new bounded subsets to $\k.$
\end{lemma}
Projection between Prikry type forcing notions arises in our work in several places. So let's present a new definition, and give an application of it.
\begin{definition}
Let $\langle \PP, \leq_\PP, \leq^*_\PP \rangle$ and $\langle \MQB, \leq_\MQB, \leq^*_\MQB \rangle$ be two forcing notions with $\leq^*_\PP \subseteq \leq_\PP$ and $\leq^*_\MQB \subseteq \leq_\MQB$. A map $\pi: \PP \rightarrow \MQB$ is a ``projection of Prikry type'' iff

$(1)$ $\pi$ is a  projection from $\langle \PP, \leq_\PP \rangle$ into $\langle \MQB, \leq_\MQB \rangle$,

$(2)$ $\pi$ preserves the $\leq^*$-relation, i.e.,  $p \leq^*_\PP q \Rightarrow \pi(p) \leq^*_\MQB \pi(q),$

$(3)$ If $p\in \PP,$  $q\in \MQB$ and $q \leq_\MQB \pi(p)$, then there exists $p^* \leq_\PP p$ such that $\pi(p^*) \leq^*_\MQB q,$

\end{definition}
It is clear that if $\pi: \PP \rightarrow \MQB$ is a  projection of Prikry type from $\PP$ into $\MQB$, then $\pi[\PP]$ is dense in $\MQB,$ with respect to both $\leq$ and $\leq^*$ relations.
Note that in the above definition we did not require $\langle \PP, \leq_\PP, \leq^*_\PP \rangle$ and $\langle \MQB, \leq_\MQB, \leq^*_\MQB \rangle$ be Prikry type forcing notions.
The following lemma shows the importance of  Prikry type projections.
\begin{lemma}
Assume  $\pi: \PP \rightarrow \MQB$  is a  projection of Prikry type, and assume $\langle \PP, \leq_\PP, \leq^*_\PP \rangle$ is of Prikry type. Then $\langle \MQB, \leq_\MQB, \leq^*_\MQB \rangle$ is also of Prikry type.
\end{lemma}
\begin{proof}
First we show that $\langle \pi[\PP], \leq_\MQB, \leq^*_\MQB \rangle$ satisfies the Prikry property.
Let $b\in R.O(\pi[\PP])$ and $q \in \pi[\PP].$ Let $p \in \PP$ be such that $q=\pi(p).$ Then  there is $p^* \leq^*_\PP p$ such that $p^*$ decides $\| b\in \pi[\dot{G}]\|_{R.O(\pi[\PP])}$ where $\dot{G}$ is the canonical name for a generic filter over $\PP.$ Let $q^*=\pi(p^*).$ Then $q^*\leq^*_\MQB q$ and  decides $b$.

But  $\pi[\PP]$ is in fact $\leq^*_\MQB-$dense in $\MQB,$ and hence
$\langle \MQB, \leq_\MQB, \leq^*_\MQB \rangle$ satisfies the Prikry property.
\end{proof}

\section{$GCH$ can fail everywhere but holds in $HOD$}
In this section we give a proof of Theorem 1.1.
Subsections 3.1-3.4 are essentially the same as in \cite{friedman-golshani}, but
we have included them
in some detail (except removing the proofs and some extra explanations) to make the paper more self-contained and as we need some of these details of the construction for later use. In subsection 3.5, we define a new forcing notion $\mathbb{Q}_{\bar{E}},$ and then in subsection 3.6 we find a Prikry type projection (defined in section 2) from $\mathbb{P}_{\bar{E}}$ into $\mathbb{Q}_{\bar{E}}.$ Then in subsection 3.7, we use the results of section 2 to prove the basic properties of a generic extension by $\mathbb{Q}_{\bar{E}},$ which are needed for our main theorem. In subsection 3.8, we prove a homogeneity result, and finally we complete the proof of the main theorem in subsection 3.9 by putting all the previous results together.
\subsection{Extender Sequences}

Suppose $j: V^{*} \rightarrow M^{*} \supseteq V_{\lambda}^{*},
crit(j)=\kappa.$ Define an extender (with projections)

\begin{center}
$E(0)= \langle \langle E_{\alpha}(0): \alpha \in \emph{A} \rangle,
\langle \pi_{\beta, \alpha}: \beta, \alpha \in \emph{A}, \beta
\geq_{j} \alpha \rangle \rangle$
\end{center}

on $\kappa$ by:

\begin{itemize}
\item $\emph{A}=[\kappa, \lambda),$ \item $\forall \alpha \in \emph{A}, E_{\alpha}(0)$ is the $\kappa-$complete ultrafilter on $\kappa$ defined by

    \begin{center}
    $X \in E_{\alpha}(0) \Leftrightarrow \alpha \in j(X)$.
    \end{center}
We write $E_\alpha(0)$ as $U_\alpha$.
    \item $\forall \alpha, \beta \in \emph{A}$
     \begin{center}
     $\beta \geq_{j} \alpha \Leftrightarrow \beta \geq \alpha$ and for some $ f \in$$ ^{\kappa} \kappa,$ $  j(f)(\beta)=\alpha$,
     \end{center}
     \item $\beta \geq_{j} \alpha \Rightarrow \pi_{\beta, \alpha}: \kappa \rightarrow \kappa$ is such that $j(\pi_{\beta, \alpha})(\beta)=\alpha$,
\item $\pi_{\alpha, \alpha}=id_{\kappa}.$
\end{itemize}
\begin{remark}
The choice of the $\pi$'s, as far as the 4-th and 5-th bullets are satisfied, does not affect the definition of the $E$-sequence
\end{remark}
Now suppose that we have defined the sequence $\langle E(\tau'): \tau' < \tau  \rangle$. If $\langle E(\tau'): \tau' < \tau  \rangle \notin M^*$ we stop the construction and set

\begin{center}
$\forall \alpha \in \emph{A}, \bar{E}_{\alpha}= \langle \alpha, E(0), ..., E(\tau'), ...: \tau'< \tau \rangle$
\end{center}
and call $\bar{E}_{\alpha}$ \emph{an extender sequence of length $\tau$}  $(\len(
\Es_{\alpha})=\tau).$

If $\langle E(\tau'): \tau' < \tau  \rangle \in M^*$  then we define
an extender (with projections)
$E(\tau)= \langle \langle E_{\langle \alpha, E(\tau'): \tau'< \tau \rangle}(\tau): \alpha \in \emph{A} \rangle,$
$ \langle \pi_{\langle \beta, E(\tau'): \tau'< \tau \rangle, \langle \alpha, E(\tau'): \tau'< \tau \rangle}: \beta, \alpha \in \emph{A}, \beta \geq_{j} \alpha \rangle \rangle$
on $V_{\kappa}$ by:

\begin{itemize}
\item $X \in E_{\langle \alpha, E(\tau'): \tau'< \tau \rangle}(\tau) \Leftrightarrow \langle \alpha, E(\tau'): \tau'< \tau \rangle \in j(X),$
\item for $\beta \geq_{j} \alpha$ in $\emph{A}, \pi_{\langle \beta, E(\tau'): \tau'< \tau \rangle, \langle \alpha, E(\tau'): \tau'< \tau \rangle}(\langle \nu, d \rangle)= \langle \pi_{\beta, \alpha}(\nu), d \rangle $
\end{itemize}

Note that $ E_{\langle \alpha, E(\tau'): \tau'< \tau \rangle}(\tau)$ concentrates on pairs of the form $\langle \nu, d \rangle$ where $\nu < \kappa$ and $d$ is an extender sequence. This makes the above definition well-defined.

We let the construction run until it stops due to the extender
sequence not being in $M^*$ or its length exceeds $j(\kappa)$.

\begin{definition}
\begin{enumerate}
\item $\bar{\mu}$ is an extender sequence if there are $j: V^{*} \rightarrow M^{*}$ and $\bar{\nu}$ such that $\bar{\nu}$ is an extender sequence derived from $j$ as above (i.e $\bar{\nu}=\bar{E_{\alpha}}$ for some $\alpha$) and $\bar{\mu}=\bar{\nu}\upharpoonright \tau$ for some $\tau \leq \len(\bar{\nu}),$

\item $\kappa(\bar{\mu})$ is the ordinal of the beginning of the sequence (i.e $\kappa(\bar{E}_{\alpha})=\alpha$),

\item $\kappa^{0}(\bar{\mu})=(\kappa(\bar{\mu}))^{0}$ (i.e $\kappa^{0}(\bar{E}_{\alpha})= \kappa=$ the critical point of $j$),

\item The sequence $ \langle \bar{\mu}_{1}, \dots, \bar{\mu}_{n} \rangle$ of extender sequences is $^{0}-$increasing if $\kappa^{0}(\bar{\mu}_1) < \dots < \kappa^{0}(\bar{\mu}_n),$

\item The extender sequence $\bar{\mu}$ is permitted to a $^{0}-$increasing sequence $ \langle \bar{\mu}_{1}, \dots, \bar{\mu}_{n} \rangle$ of extender sequences if $\kappa^{0}(\bar{\mu}_n)<\kappa^{0}(\bar{\mu}),$

\item Notation: We write $X \in \bar{E}_{\alpha}$ iff $\forall \xi < \len(\bar{E}_{\alpha}), X \in E_{\alpha}(\xi),$

\item $\bar{E}= \langle \bar{E}_{\alpha}: \alpha \in A  \rangle$ is an \emph{extender sequence system} if there is $j: V^{*} \rightarrow M^{*}$ such that each $\bar{E}_{\alpha}$ is derived from $j$ as above and $\forall \alpha, \beta \in A, \len(\bar{E}_{\alpha})= \len(\bar{E}_{\beta}).$ Call this common length, the length of $\bar{E}, \len(\bar{E}),$

\item For an extender sequence $\bar{\mu},$ we use $\bar{E}(\bar{\mu})$ for the extender sequence system containing $\bar{\mu}$ (i.e $\bar{E}(\bar{E}_{\alpha})= \bar{E}$),

\item $\dom(\bar{E})=A$,

\item $\bar{E}_{\beta}  \geq_{\bar{E}} \bar{E}_{\alpha} \Leftrightarrow \beta  \geq_{j} \alpha.$
\end{enumerate}
\end{definition}

\subsection{Finding generic filters}

Using $GCH$ in $V^*$ we construct  an extender sequence system $\bar{E}= \langle \bar{E}_{\alpha}: \alpha \in \dom\bar{E} \rangle$ where $\dom\bar{E}=[\kappa, \kappa^{+3})$ and $\len(\bar{E})=\kappa^{+}$ such that
the ultrapower $j_{\bar{E}}:V^{*}\rightarrow M_{\bar{E}}^{*}$ (defined
below) contains $V_{\kappa+3}^{*}.$ Suppose that $\bar{E}$ is derived from a $(\kappa+4)$-strong embedding $j: V^{*} \rightarrow M^{*}.$   Consider the following elementary embeddings $\forall \tau' < \tau < \len(\bar{E})=\kappa^+:$
\begin{align*} \label{E-system}
& j_\gt\func  \VS \to \MSt \simeq \Ult(\VS, E(\gt))=
\{j_\tau(f)(\bar E_\alpha \restricted \tau)\mid f\in V^*\},
\notag \\
&  k_\gt(j_\gt(f)(\Es_\ga \restricted \gt))=
        j(f)(\Es_\ga \restricted \gt),
\\
\notag & i_{\gt', \gt}(j_{\gt'}(f)(\Es_\ga \restricted \gt')) =
    j_\gt(f)(\Es_\ga \restricted \gt'),
\\
\notag & \ordered{\MSE,i_{\gt, \Es}} = \limdir \ordered {
        \ordof{\MSt} {\gt < \len(\Es)},
                \ordof{i_{\gt',\gt}} {\gt' \leq \gt < \len(\Es)}
        }.
\end{align*}
We demand that
        $\Es \restricted \gt \in \MSt$ for all $\tau<\len(\bar E)$.

Thus we get the following commutative diagram.

\begin{align*}
\begin{diagram}
\node{\VS}
        \arrow[3]{e,t}{j}
        \arrow{sse,l}{j_{\gt'}}
        \arrow[2]{se,l}{j_\gt}
        \arrow{seee,t,l}{j_{\Es}}
    \node{}
    \node{}
    \node{\MS}
\\
\node{}
    \node{}
    \node{}
        \node{\MSE}
        \arrow{n,r}{k_{\Es}}
\\
\node{}
    \node{M^*_{\gt'}}
        \arrow[2]{ne,t,3}{k_{\gt'}}
        \arrow{nee,t,2}{i_{\gt', \Es}}
        \arrow{e,b}{i_{\gt', \gt}}
    \node{\MSt = \Ult(\VS, E(\gt))}
        \arrow[1]{ne,b}{i_{\gt, \Es}}
        \arrow{nne,b,1}{k_{\gt}}
\end{diagram}
\end{align*}
Note that
\begin{center}
 $ \kappa^{+4}_{M_{\tau^{'}}^{*}} < j_{\tau^{'}}(\kappa) < \kappa^{+4}_{M_{\tau}^{*}} < j_{\tau}(\kappa) < \kappa^{+4}_{M_{\bar{E}^{*}}} \leq \kappa^{+4}_{M^{*}} \leq \kappa^{+4}.$
\end{center}
We also factor through the normal ultrafilter $E_\kappa(0)$ on $\kappa$ to get the following commutative diagram

\begin{align*}
\begin{aligned}
\begin{diagram}
\node{\VS}
        \arrow{e,t}{j_\Es}
        \arrow{se,t}{j_\gt}
        \arrow{s,l}{i_U}
        \node{\MSE}
\\
\node{\NS \simeq \Ult(\VS, U)}
         \arrow{e,b}{i_{U, \gt}}
         \arrow{ne,b}{i_{U, \Es}}
        \node{\MSt}
         \arrow{n,b}{i_{\gt, \Es}}
\end{diagram}
\end{aligned}
\begin{aligned}
\qquad
\begin{split}
& U = E_\gk(0),
\\
& i_U \func  \VS \to \NS \simeq \Ult(\VS, U),
\\
& i_{U, \gt}(i_U(f)(\gk)) = j_\gt(f)(\gk),
\\
& i_{U, \Es}(i_U(f)(\gk)) = j_\Es(f)(\gk).
\end{split}
\end{aligned}
\end{align*}
$\NS$ catches $\VS$ only up to $\VS_{\gk+1}$ and we have
\begin{align*}
\gk^+ < \crit i_{U, \gt} = \crit i_{U, \Es}
    = \gk^{++}_{\NS} < i_U(\gk) < \gk^{++}.
\end{align*}

We now define the forcings for which we will need ``guiding
generics''.

\begin{definition}
Let
\begin{enumerate}
\item $\mathbb{R}_{U}^{\Col}=\Col(\kappa^{+6}, i_{U}(\kappa))_{N^{*}},$

\item  $\mathbb{R}_{U}^{\Add, 1}= \Add(\kappa^{+}, \kappa^{+4})_{N^{*}},$

\item  $\mathbb{R}_{U}^{\Add, 2}= \Add(\kappa^{++}, \kappa^{+5})_{N^{*}},$
\item  $\mathbb{R}_{U}^{\Add, 3}=  \Add(\kappa^{+3}, \kappa^{+6})_{N^{*}},$
\item  $\mathbb{R}_{U}^{\Add, 4}= (\Add(\kappa^{+4}, i_{U}(\kappa)^{+})  \times \Add(\kappa^{+5}, (i_{U}(\kappa)^{++})_{N^{*2}})  \times \Add(\kappa^{+6}, (i_{U}(\kappa)^{+3})_{N^{*2}}) )_{N^{*}},$ where $N^{*2}$ is the second iterate of $V^*$ by $U$,

\item $\mathbb{R}_{U}^{\Add}=\mathbb{R}_{U}^{\Add,1} \times \mathbb{R}_{U}^{\Add,2} \times \mathbb{R}_{U}^{\Add,3} \times \mathbb{R}_{U}^{\Add,4},$

\item $\mathbb{R}_{U} = \mathbb{R}_{U}^{\Add} \times \mathbb{R}_{U}^{\Col}.$
\end{enumerate}
\end{definition}
\begin{remark}
$(j_{\tau}(\kappa)^{++})_{M^*_{\tau}}=(j_{\tau}(\kappa)^{++})_{M^{*2}_{\tau}}$ and $(j_{\tau}(\kappa)^{+3})_{M^*_{\tau}}=(j_{\tau}(\kappa)^{+3})_{M^{*2}_{\tau}}$, where $M^{*2}_{\tau}$ is the second iterate of $V^*$ by $E(\tau).$ Similarly
$(j_{\bar{E}}(\kappa)^{++})_{M^*_{\bar{E}}}=(j_{\bar{E}}(\kappa)^{++})_{M^{*2}_{\bar{E}}}$ and $(j_{\bar{E}}(\kappa)^{+3})_{M^*_{\bar{E}}}=(j_{\bar{E}}(\kappa)^{+3})_{M^{*2}_{\bar{E}}}$, where $M^{*2}_{\bar{E}}$ is the second iterate of $V^*$ by $\bar{E}.$
\end{remark}
Similarly define the forcing notions $\mathbb{R}_\tau$ and $\mathbb{R}_{\bar{E}},$ where $i_U(\kappa), N^*$ are replaced by $j_\tau(\kappa), M^*_\tau$
and $j_{\bar{E}}(\kappa), M^*_{\bar{E}}$ respectively.
Also define the forcing notion $\mathbb{P}$ as follows
\begin{center}
 $\mathbb{P}=\mathbb{P}_{1} \times \mathbb{P}_{2} \times \mathbb{P}_{3} = \Add(\kappa^{+}, (\kappa^{+4})_{M^*_{\bar{E}}}) \times \Add(\kappa^{++}, (\kappa^{+5})_{M^*_{\bar{E}}}) \times \Add(\kappa^{+3}, (\kappa^{+6})_{M^*_{\bar{E}}}) $\footnote{ Hence $\mathbb{P}$ is forcing isomorphic to $\Add(\kappa^{+}, \kappa^{+4}) \times \Add(\kappa^{++}, \kappa^{+4}) \times \Add(\kappa^{+3}, \kappa^{+4})$.}
\end{center}
and let $G=G_{1} \times G_{2} \times G_{3}$ be $\mathbb{P}-$generic over $V^*$. It is clear that $V^{*}[G]$ is a cofinality-preserving generic extension of $V^{*}$ and that $GCH$ holds in $V^{*}[G]$ below and at $\kappa.$ The forcing $\mathbb P$ is our ``preparation
forcing'' (which preserves the $GCH$ below $\kappa$ and facilitates the construction of guiding generics). We set $V=V^*[G]$.

\begin{lemma}
$(1)$ $G_{U}= \langle i_{U}^{''} G_{1} \rangle \times \langle i_{U}^{''} G_{2} \rangle \times \langle i_{U}^{''} G_{3} \rangle$ is $\mathbb{P}_{U}=i_{U}(\mathbb{P})-$generic over $N^{*}$,

$(2)$ $G_{\tau}= \langle j_{\tau}^{''} G_{1} \rangle \times \langle j_{\tau}^{''} G_{2} \rangle \times \langle j_{\tau}^{''} G_{3} \rangle$ is $\mathbb{P}_{\tau}=j_{\tau}(\mathbb{P})-$generic over $M_{\tau}^{*}$,

$(3)$ $G_{\bar{E}}= \langle \bigcup_{\tau < \len(\bar{E})} i_{\tau, \bar{E}}^{''}G_{\tau} \rangle $ is $\mathbb{P}_{\bar{E}}=j_{\bar{E}}(\mathbb{P})-$generic over $M_{\bar{E}}^{*}$,.
\end{lemma}
The generic filters above are such that
 the following diagram is well-defined and commutative:

\begin{align*}
\begin{diagram}
\node{V = \VS[G]}
        \arrow[2]{e,t}{j_\Es}
        \arrow{s,l}{i_{U}}
        \arrow{se,b}{j_{\gt'}}
        \arrow{see,b}{j_\gt}
    \node{}
        \node{\ME = \MSE[G_\Es]}
\\
    \node{N = \NS[G_U]}
         \arrow{e,b}{i_{U, \gt'}}
    \node{M_{\gt'} = M^*_{\gt'}[G_{\gt'}]}
         \arrow{ne,t,3}{i_{\gt', \Es}}
         \arrow{e,b}{i_{\gt', \gt}}
    \node{\Mt = \MSt[G_\gt]}
        \arrow[1]{n,b}{i_{\gt, \Es}}
\end{diagram}
\end{align*}

\begin{lemma} (Existence of guiding generics)
In $V^{*}[G]$ there are $I_{U}, I_{\tau}$ and $I_{\bar{E}}$ such that

$(1)$ $I_{U}$ is $R_{U}-$generic over $N^{*}[G_{U}]$,

$(2)$ $I_{\tau}$ is $R_{\tau}-$generic over $M_{\tau}^{*}[G_{\tau}]$,

$(3)$ $I_{\bar{E}}$ is $R_{\bar{E}}-$generic over $M_{\bar{E}}^{*}[G_{\bar{E}}]$,

$(4)$ The generics are so that we have the following lifting diagram

\begin{align*}
\begin{diagram}
\node{}
    \node{}
        \node{\ME[I_\Es]}
\\
    \node{N[I_U]}
         \arrow{e,b}{i^*_{U, \gt'}}
    \node{M_{\gt'}[I_{\gt'}]}
         \arrow{ne,t,3}{i^*_{\gt', \Es}}
         \arrow{e,b}{i^*_{\gt', \gt}}
    \node{\Mt[I_\gt]}
        \arrow[1]{n,b}{i^*_{\gt, \Es}}
\end{diagram}
\end{align*}
\end{lemma}

Let $R(-,-)$ be a function such that
\begin{center}
$i_{U}^{2}(R)(\kappa, i_{U}(\kappa))=\mathbb{R}_{U},$
\end{center}
where $i_{U}^{2}$ is the second iterate of $i_{U}.$
By applying $i^2_{U, \bar{E}}$ we get.
\begin{lemma}
 $j_{\bar{E}}^{2}(R)(\kappa, j_{\bar{E}}(\kappa))=\mathbb{R}_{\bar{E}}.$
\end{lemma}

\subsection{Redefining extender Sequences}
As in \cite{merimovich2}, in the prepared model $V=V^*[G]$  we define a new extender sequence system $\bar{F}= \langle \bar{F}_{\alpha}: \alpha \in \dom(\bar{F})\rangle$ by:
\begin{itemize}
  \item $\dom(\bar{F})=\dom(\bar{E}),$
  \item $\len(\bar{F})=\len(\bar{E})$
  \item $\leq_{\bar{F}}=\leq_{\bar{E}},$ \item $F(0)=E(0),$
  \item $I(\tau)=I_{\tau}$ (the guiding $\mathbb{R}_\tau$-generic over $M_\tau=M^*_\tau[G_\tau]$),
  \item $\forall 0< \tau < \len(\bar{F}), F(\tau)= \langle \langle F_{\alpha}(\tau): \alpha \in \dom(\bar{F}) \rangle, \langle \pi_{\beta, \alpha}: \beta, \alpha \in \dom(\bar{F}), \beta \geq_{\bar{F}} \alpha \rangle \rangle$  is such that
    \begin{center}
    $X \in F_{\alpha}(\tau) \Leftrightarrow \langle \alpha, F(0), I(0), ..., F(\tau^{'}), I(\tau^{'}), ...: \tau^{'}  < \tau \rangle \in j_{\bar{E}}(X),$
    \end{center}
and
\begin{center}
$\pi_{\beta, \alpha}(\langle \xi, d \rangle)= \langle \pi_{\beta, \alpha}(\xi), d \rangle, $
\end{center}
\item $\forall \alpha \in \dom(\bar{F}), \bar{F}_{\alpha}= \langle \alpha, F(\tau),I(\tau): \tau < \len(\bar{F})  \rangle.$
\end{itemize}
Also let $I(\bar{F})$ be the filter generated by $\bigcup_{\tau < \len(\bar{F})} i_{\tau, \bar{E}}^{''}I(\tau).$ Then $I(\bar{F})$ is $\mathbb{R}_{\bar{E}}-$generic over $M_{\bar{E}}.$

From now on we work with this new definition of extender sequence
system and use $\bar{E}$ to denote it.

\begin{definition}$(1)$ We write $T \in \bar{E}_{\alpha}$ iff $\forall \xi < \len(\bar{E}_{\alpha}), T \in E_{\alpha}(\xi),$

$(2)$ $T \backslash \bar{\nu} = T \backslash V_{\kappa^{0}(\bar{\nu})}^{*},$

$(3)$ $T \upharpoonright \bar{\nu}= T \cap V_{\kappa^{0}(\bar{\nu})}^{*}.$

\end{definition}

We now define two forcing notions $\mathbb{P}_{\bar{E}}$ and $\MQB_{\bar{E}}.$

\subsection{Definition of the forcing notion $\mathbb{P}_{\bar{E}}$}
 This forcing notion, defined in the ground model $V=V^*[G]$, is  the forcing notion of \cite{friedman-golshani}. We give it in detail for completeness and later use. First we define a forcing notion $\mathbb{P}_{\bar{E}}^{*}.$

\begin{definition}
A condition $p$ in $\mathbb{P}_{\bar{E}}^{*}$ is of the form
\begin{center}
$p = \{ \langle \bar{\gamma}, p^{\bar{\gamma}}\rangle: \bar{\gamma} \in s \} \cup \{\langle \bar{E}_{\alpha}, T, f, F \rangle  \}$
\end{center}
where
\begin{enumerate}
\item $s \in [\bar{E}]^{\leq \kappa}, \min\bar{E}= \bar{E}_{\kappa} \in s,$

 \item$p^{\Es_{\kappa}} \in V_{\kappa^{0}(\bar{E})}^{*}$ is an extender sequence such that $\kappa(p^{\bar{E}_{\kappa}})$ is inaccessible ( we allow $p^{\bar{E}_{\kappa}}= \emptyset$). Write $p^0$ for $p^{\bar{E}_{\kappa}}.$

\item $\forall \bar{\gamma} \in s \backslash \{ \min(s) \}, p^{\bar{\gamma}} \in [V_{\kappa^{0}(\bar{E})}^{*}] ^{< \omega}$ is a $^{0}$-increasing sequence of extender sequences and $\max\kappa(p^{\bar{\gamma}})$ is inaccessible,

\item $\forall \bar{\gamma} \in s, \kappa(p^{0}) \leq \max\kappa(p^{\bar{\gamma}})$,

\item $\forall \bar{\gamma} \in s, \bar{E}_{\alpha} \geq_{\bar{E}} \bar{\gamma},$

\item $T \in \bar{E}_{\alpha},$

\item $\forall \bar{\nu} \in T, \mid \{ \bar{\gamma} \in s: \bar{\nu}$ is permitted to $p^{\bar{\gamma}} \} \mid \leq \kappa^{0}(\bar{\nu}),$

\item $\forall \bar{\beta}, \bar{\gamma} \in s, \forall \bar{\nu} \in T,$ if $\bar{\beta} \neq \bar{\gamma}$ and $\bar{\nu}$ is permitted to $p^{\bar{\beta}}, p^{\bar{\gamma}},$ then $\pi_{\bar{E}_{\alpha}, \bar{\beta}}(\bar{\nu}) \neq \pi_{\bar{E}_{\alpha}, \bar{\gamma}}(\bar{\nu}),$

\item $f$ is a function such that

$\hspace{.5cm}$ $(9.1)$ $\dom(f)= \{\bar{\nu} \in T: \len(\bar{\nu})=0 \},$

$\hspace{.5cm}$ $(9.2)$ $f(\nu_{1}) \in R(\kappa(p^{0}), \nu_{1}^{0}).$ If $p^{0}=\emptyset,$ then $f(\nu_{1})= \emptyset,$

\item $F$ is a function such that

$\hspace{.5cm}$ $(10.1)$ $\dom(F)= \{ \langle \bar{\nu_{1}}, \bar{\nu_{2}} \rangle \in T^{2}:\len(\bar{\nu_{1}})=\len(\bar{\nu_{2}})= \emptyset \},$

$\hspace{.5cm}$ $(10.2)$ $F(\nu_{1}, \nu_{2}) \in R(\nu_{1}^{0}, \nu_{2}^{0}),$

$\hspace{.5cm}$ $(10.3)$ $j_{\bar{E}}^{2}(F)(\alpha, j_{\bar{E}}(\alpha)) \in I(\bar{E})$.
\end{enumerate}
\end{definition}
We write $\mc(p), \supp(p), T^{p}, f^{p}$ and $F^{p}$ for $\bar{E}_{\alpha}, s, T, f$ and $F$ respectively.
\begin{definition}
For $p, q \in \mathbb{P}_{\bar{E}}^{*},$ we say $p$ is a Prikry extension of $q$ ($p \leq^{*} q$ or $p \leq^{0} q$) iff
\begin{enumerate}
\item $\supp(p) \supseteq \supp(q),$

\item  $\forall \bar{\gamma} \in \supp(q), p^{\bar{\gamma}}=q^{\bar{\gamma}},$

\item  $\mc(p) \geq_{\bar{E}} \mc(q),$

\item  $\mc(p) >_{\bar{E}} \mc(q) \Rightarrow \mc(q) \in \supp(p),$

\item  $\forall \bar{\gamma} \in \supp(p) \backslash \supp(q), \max\kappa^{0}(p^{\bar{\gamma}}) > \bigcup \bigcup j_{\bar{E}}(f^{q})(\kappa(\mc(q))),$

\item  $T^{p} \leq \pi_{\mc(p), \mc(q)}^{-1''}T^{q},$

\item  $\forall \bar{\gamma} \in \supp(q), \forall \bar{\nu} \in T^{p},$ if $\bar{\nu}$ is permitted to $p^{\bar{\gamma}},$ then
\begin{center}
$\pi_{\mc(p), \bar{\gamma}}(\bar{\nu})=\pi_{\mc(q), \bar{\gamma}}(\pi_{\mc(p), \mc(q)}(\bar{\nu})),$
\end{center}

\item  $\forall \nu_{1} \in \dom(f^{p}), f^{p}(\nu_{1}) \leq f^{q}\circ\pi_{\mc(p), \mc(q)}(\nu_{1}),$

\item  $\forall \langle \nu_{1}, \nu_{2} \rangle \in \dom(F^{p}), F^{p}(\nu_{1}, \nu_{2}) \leq F^{q}\circ \pi_{\mc(p), \mc(q)}(\nu_{1}, \nu_{2}).$
\end{enumerate}

\end{definition}

We are now ready to define the forcing notion $\mathbb{P}_{\bar{E}}.$
\begin{definition}
A condition $p$ in $\mathbb{P}_{\bar{E}}$ is of the form
\begin{center}
$p=p_{n} ^{\frown} ...^{\frown} p_{0} $
\end{center}
where
\begin{itemize}
\item $p_{0} \in \mathbb{P}_{\bar{E}}^{*}, \kappa^{0}(p_{0}^{0}) \geq \kappa^{0}(\bar{\mu}_{1}),$ \item $p_{1} \in \mathbb{P}_{\bar{\mu}_{1}}^{*}, \kappa^{0}(p_{1}^{0}) \geq \kappa^{0}(\bar{\mu}_{2}),$

$\vdots$

    \item $p_{n} \in \mathbb{P}_{\bar{\mu}_{n}}^{*}.$
\end{itemize}
and $\langle \bar{\mu}_{n}, ..., \bar{\mu}_{1}, \bar{E} \rangle$ is a $^{0}-$inceasing sequence of extender sequence systems, that is $\kappa^{0}(\bar{\mu}_{n}) < ... < \kappa^{0}(\bar{\mu}_{1}) < \kappa^{0}(\bar{E}).$
\end{definition}
\begin{definition}
For $p, q \in \mathbb{P}_{\bar{E}},$ we say $p$ is a Prikry extension of $q$ ($p \leq^{*} q$ or $p \leq^{0} q$) iff
\begin{center}
$p=p_{n} ^{\frown} ...^{\frown} p_{0} $

$q=q_{n} ^{\frown} ...^{\frown} q_{0} $
\end{center}
where
\begin{itemize}
\item $p_{0}, q_{0} \in \mathbb{P}_{\bar{E}}^{*}, p_{0} \leq^{*} q_{0},$  \item $p_{1}, q_{1} \in \mathbb{P}_{\bar{\mu}_{1}}^{*}, p_{1} \leq^{*} q_{1},$

   $\vdots$

     \item $p_{n}, q_{n} \in \mathbb{P}_{\bar{\mu}_{n}}^{*}, p_{n} \leq^{*} q_{n}.$
\end{itemize}
\end{definition}
Now let $p \in \mathbb{P}_{\bar{E}}$ and $\bar{\nu} \in T^{p}.$ We define $p_{\langle \bar{\nu} \rangle}$ a one element extension of $p$ by $\bar{\nu}.$
\begin{definition}
Let $p \in \mathbb{P}_{\bar{E}}, \bar{\nu} \in T^{p}$ and $\kappa^{0}(\bar{\nu}) > \bigcup \bigcup j_{\bar{E}}(f^{p, \Col})(\kappa(\mc(p)))$, where $f^{p, \Col}$ is the collapsing part of $f^{p}$. Then $p_{\langle \bar{\nu}\rangle}=p_{1} ^{\frown} p_{0}$ where
\begin{enumerate}
\item $\supp(p_{0})=\supp(p),$

\item $\forall \bar{\gamma} \in \supp(p_{0}),$

$p_{0}^{\bar{\gamma}} = \left\{
\begin{array}{l}
      \pi_{\mc(p), \bar{\gamma}}(\bar{\nu}) \hspace{1.65cm} \text{ if } \bar{\nu} \text{ is permitted to } p^{\bar{\gamma}} \text{ and } \len(\bar{\nu}) >0, \\
       \pi_{\mc(p), \bar{\gamma}}(\bar{\nu}) \hspace{1.65cm} \text{ if } \bar{\nu} \text{ is permitted to } p^{\bar{\gamma}}, \len(\bar{\nu})=0 \text{ and } \bar{\gamma}=\bar{E}_{\kappa},
       \\
       p^{\bar{\gamma} \frown} \langle \pi_{\mc(p), \bar{\gamma}}(\bar{\nu}) \rangle \hspace{.7cm} \text{ if } \bar{\nu} \text{ is permitted to } p^{\bar{\gamma}}, \len(\bar{\nu})=0 \text{ and } \bar{\gamma}\neq \bar{E}_{\kappa},
       \\
       p^{\bar{\gamma}} \hspace{2.95cm} \text{ otherwise }.

     \end{array} \right.$

\item $\mc(p_{0})=\mc(p),$

\item $T^{p_{0}}=T^{p} \backslash \bar{\nu},$

\item $\forall \nu_{1} \in T^{p_{0}}, f^{p_{0}}(\nu_{1})=F^{p}(\kappa(\bar{\nu}), \nu_{1}),$

\item $F^{p_{0}}=F^{p},$

\item if $\len(\bar{\nu})>0$ then

$\hspace{.5cm}$ $(7.1)$ $\supp(p_{1})=\{\pi_{\mc(p), \bar{\gamma}}(\bar{\nu}): \bar{\gamma} \in \supp(p)$ and $\bar{\nu}$ is permitted to $p^{\bar{\gamma}}\},$

$\hspace{.5cm}$ $(7.2)$ $p_{1}^{\pi_{\mc(p), \bar{\gamma}}(\bar{\nu})}=p^{\bar{\gamma}},$

$\hspace{.5cm}$ $(7.3)$ $\mc(p_{1})=\bar{\nu},$

$\hspace{.5cm}$ $(7.4)$ $T^{p_{1}}=T^{p} \upharpoonright \bar{\nu},$

$\hspace{.5cm}$ $(7.5)$ $f^{p_{1}}=f^{p} \upharpoonright \bar{\nu},$

$\hspace{.5cm}$ $(7.6)$ $F^{p_{1}}=F^{p} \upharpoonright \bar{\nu},$

\item if $\len(\bar{\nu})=0$ then

$\hspace{.5cm}$ $(8.1)$ $\supp{p_{1}} =\{ \pi_{\mc(p),0}(\bar{\nu}) \},$

$\hspace{.5cm}$ $(8.2)$ $p_{1}^{\pi_{\mc(p),0}(\bar{\nu})}=p^{\bar{E}_{\kappa}},$

$\hspace{.5cm}$ $(8.3)$ $\mc(p_{1})=\bar{\nu}^{0},$

$\hspace{.5cm}$ $(8.4)$ $T^{p_{1}}= \emptyset,$

$\hspace{.5cm}$ $(8.5)$ $f^{p_{1}}=f^{p}(\kappa(\bar{\nu})),$

$\hspace{.5cm}$ $(8.6)$ $F^{p_{1}}= \emptyset.$
\end{enumerate}
\end{definition}
We use $(p_{\langle \bar{\nu} \rangle})_{0}$ and $(p_{\langle \bar{\nu} \rangle})_{1}$ for $p_{0}$ and $p_{1}$ respectively. We also let $p_{\langle \bar{\nu_{1}}, \bar{\nu_{2}} \rangle }= (p_{\langle \bar{\nu}_{1}\rangle})_{1} ^{\frown} (p_{\langle \bar{\nu}_{1} \rangle})_{0 \langle \bar{\nu_{2}} \rangle}$ and so on.

The above definition is the key step in the definition of the forcing relation $\leq.$
\begin{definition}
For $p, q \in \mathbb{P}_{\bar{E}},$ we say $p$ is a $1-$point extension of $q$ ($p \leq^{1} q$) iff
\begin{center}
$p=p_{n+1} ^{\frown} ...^{\frown} p_{0} $

$q=q_{n} ^{\frown} ...^{\frown} q_{0} $
\end{center}
and there is $0 \leq k \leq n$ such that
\begin{itemize}
\item $\forall i < k, p_{i}, q_{i} \in \mathbb{P}_{\bar{\mu}_{i}}^{*}, p_{i} \leq^{*} q_{i},$ \item $\exists \bar{\nu} \in T^{q_{k}}, (p_{k+1}) ^{\frown}p_{k} \leq^{*} (q_{k})_{ \langle \bar{\nu} \rangle}$ \item $\forall i > k, p_{i+1}, q_{i} \in \mathbb{P}_{\bar{\mu}_{i}}^{*}, p_{i+1} \leq^{*} q_{i},$
\end{itemize}
where $\bar{\mu}_{0}=\bar{E}.$
\end{definition}
\begin{definition}
For $p, q \in \mathbb{P}_{\bar{E}},$ we say $p$ is an $n-$point extension of $q$ ($p \leq^{n} q$) iff there are $p^{n}, ..., p^{0}$ such that
\begin{center}
$p=p^{n} \leq^{1} ... \leq^{1} p^{0}=q.$
\end{center}
\end{definition}
\begin{definition}
For $p, q \in \mathbb{P}_{\bar{E}},$ we say $p$ is an extension of $q$ ($p \leq q$) iff there is some $n$ such that $p \leq^{n} q$.
\end{definition}
Suppose that $H$ is $\mathbb{P}_{\bar{E}}-$generic over $V=V^*[G]$. For $\alpha \in \dom(\bar{E})$ set
\begin{center}
$C_{H}^{\alpha} = \{ \max\kappa(p_{0}^{\bar{E}_{\alpha}}): p \in H \}.$
\end{center}
The following theorem summarizes the main properties of the forcing  extension.
\begin{theorem}
$(1)$ $V[H]$ and $V$ have the same cardinals $\geq \kappa,$

$(2)$ $\kappa$ remains strongly inaccessible in $V[H]$

$(3)$ $C_{H}^{\alpha}$ is unbounded in $\kappa,$

$(4)$ $C_{H}^{\kappa}$ is a club in $\kappa,$

$(5)$ $\alpha \neq \beta \Rightarrow C_{H}^{\alpha} \neq C_{H}^{\beta},$

$(6)$ Let $\l=\min(C_{H}^{\kappa}),$ and let $K$ be $\Col(\omega, \l^{+})_{V[H]}-$generic over $V[H].$ Then
\begin{center}
$\C$$ARD^{V[H][K]} \cap\kappa = (\lim(C_{H}^{\kappa}) \cup \{\mu^{+}, ..., \mu^{+6}: \mu \in C_{H}^{\kappa}\} \backslash \l^{++}) \cup \{ \omega \},$
\end{center}

$(7)$ $V[H][K] \models  `` \forall \lambda \leq \kappa, 2^{\lambda}=\lambda^{+3}$''.
\end{theorem}

\subsection{Definition of the forcing notion $\mathbb{Q}_{\bar{E}}$}
 We now define another forcing notion $\mathbb{Q}_{\bar{E}}.$ This forcing notion will produce a cardinal and $GCH$ preserving extension of $V$ such that there exists a projection $\pi: \mathbb{P}_{\bar{E}} \rightarrow \mathbb{Q}_{\bar{E}},$
  and $\mathbb{Q}_{\bar{E}}$ is as close to $\mathbb{P}_{\bar{E}}$ as possible, so that the quotient forcing  is sufficiently homogeneous to guarantee that $HOD^{V[H]} \subseteq V[H^*],$ where $H^*$ is the filter generated by $\pi[H].$

The definition of  $\mathbb{Q}_{\bar{E}}$ is similar to that of  $\mathbb{P}_{\bar{E}},$ and so we proceed the definitions analogous to those of the last section.
First we define a forcing notion $\mathbb{Q}_{\bar{E}}^{*}.$

\begin{definition}
A condition $p$ in $\mathbb{Q}_{\bar{E}}^{*}$ is of the form
\begin{center}
$p = \{ \langle \bar{\gamma}, p^{\bar{\gamma}}\rangle: \bar{\gamma} \in s \} \cup \{\langle \bar{E}_{\alpha}, T \rangle  \}$
\end{center}
where
\begin{enumerate}
\item $s \in [\bar{E}]^{\leq \kappa}, \min\bar{E}= \bar{E}_{\kappa} \in s,$

 \item$p^{\Es_{\kappa}} \in V_{\kappa^{0}(\bar{E})}^{*}$ is an extender sequence such that $\kappa(p^{\bar{E}_{\kappa}})$ is inaccessible ( we allow $p^{\bar{E}_{\kappa}}= \emptyset$). Write $p^0$ for $p^{\bar{E}_{\kappa}}.$

\item $\forall \bar{\gamma} \in s \backslash \{ \min(s) \}, p^{\bar{\gamma}}=\langle \rangle,$

\item $\forall \bar{\gamma} \in s, \bar{E}_{\alpha} \geq_{\bar{E}} \bar{\gamma},$

\item $T \in \bar{E}_{\alpha},$

\end{enumerate}
\end{definition}
We write $\mc(p), \supp(p)$ and  $T^{p}$, for $\bar{E}_{\alpha}, s$ and  $T$ respectively.
\begin{definition}
For $p, q \in \mathbb{Q}_{\bar{E}}^{*},$ we say $p$ is a Prikry extension of $q$ ($p \leq^{*} q$ or $p \leq^{0} q$) iff
\begin{enumerate}
\item $\supp(p) \supseteq \supp(q),$

\item  $p^0=q^0$

\item  $\mc(p) \geq_{\bar{E}} \mc(q),$

\item  $\mc(p) >_{\bar{E}} \mc(q) \Rightarrow \mc(q) \in \supp(p),$

\item  $T^{p} \leq \pi_{\mc(p), \mc(q)}^{-1''}T^{q},$
\end{enumerate}

\end{definition}
Before we continue, let us mention that in the definition of a condition $p$ in $\mathbb{Q}_{\bar{E}}^{*},$ the set $\supp(p)$
has no real role, and we can avoid mentioning it; all we require is to know what $\mc(p)$ is. We can make this precise as follows.
Define a relation $\sim$ on $\mathbb{Q}_{\bar{E}}^{*}$ by $p \sim q$ if and only if the following conditions are satisfied:
\begin{enumerate}
\item [$\sim_1:$] $\langle \bar{E}_\kappa, p^{\bar{E}_\kappa} \rangle= \langle \bar{E}_\kappa, q^{\bar{E}_\kappa} \rangle,$

\item [$\sim_2:$] $\mc(p)=\mc(q),$

\item [$\sim_3:$] $T^p=T^q.$
\end{enumerate}
It is easily seen that $\sim$ is an equivalence relation. We identify two elements of $\mathbb{Q}_{\bar{E}}^{*}$ if they are $\sim$-equivalent.
In what follows we use our original definition given above, but we have in mind that a condition should be replaced with its equivalence class. We may note that this identification is used, for example, in the proof of Lemmas 3.28 and 3.36$(4)$ below.
We now  define the forcing notion $\mathbb{Q}_{\bar{E}}.$
\begin{definition}
A condition $p$ in $\mathbb{Q}_{\bar{E}}$ is of the form
\begin{center}
$p=p_{n} ^{\frown} ...^{\frown} p_{0} $
\end{center}
where
\begin{itemize}
\item $p_{0} \in \mathbb{Q}_{\bar{E}}^{*}, \kappa^{0}(p_{0}^{0}) \geq \kappa^{0}(\bar{\mu}_{1}),$
\item $p_{1} \in \mathbb{Q}_{\bar{\mu}_{1}}^{*}, \kappa^{0}(p_{1}^{0}) \geq \kappa^{0}(\bar{\mu}_{2}),$

$\vdots$

    \item $p_{n} \in \mathbb{Q}_{\bar{\mu}_{n}}^{*}.$
\end{itemize}
and $\langle \bar{\mu}_{n}, ..., \bar{\mu}_{1}, \bar{E} \rangle$ is a $^{0}-$inceasing sequence of extender sequence systems, that is $\kappa^{0}(\bar{\mu}_{n}) < ... < \kappa^{0}(\bar{\mu}_{1}) < \kappa^{0}(\bar{E}).$
\end{definition}
\begin{definition}
For $p, q \in \mathbb{Q}_{\bar{E}},$ we say $p$ is a Prikry extension of $q$ ($p \leq^{*} q$ or $p \leq^{0} q$) iff
\begin{center}
$p=p_{n} ^{\frown} ...^{\frown} p_{0} $

$q=q_{n} ^{\frown} ...^{\frown} q_{0} $
\end{center}
where
\begin{itemize}
\item $p_{0}, q_{0} \in \mathbb{Q}_{\bar{E}}^{*}, p_{0} \leq^{*} q_{0},$
\item $p_{1}, q_{1} \in \mathbb{Q}_{\bar{\mu}_{1}}^{*}, p_{1} \leq^{*} q_{1},$

   $\vdots$

     \item $p_{n}, q_{n} \in \mathbb{Q}_{\bar{\mu}_{n}}^{*}, p_{n} \leq^{*} q_{n}.$
\end{itemize}
\end{definition}
Now let $p \in \mathbb{Q}_{\bar{E}}$ and $\bar{\nu} \in T^{p}.$ We define $p_{\langle \bar{\nu} \rangle}$ a one element extension of $p$ by $\bar{\nu}.$
\begin{definition}
Let $p \in \mathbb{Q}_{\bar{E}}$ and  $\bar{\nu} \in T^{p}$. Then $p_{\langle \bar{\nu}\rangle}=p_{1} ^{\frown} p_{0}$ where
\begin{enumerate}
\item $\supp(p_{0})=\supp(p),$

\item $p_0^0= \pi_{\mc(p),0}(\bar{\nu}),$

\item
$p_{0}^{\bar{\gamma}} = \langle \rangle$ if $\bar{\gamma} \neq \bar{E}_\kappa,$

\item $T^{p_0}=T^p \setminus \bar{\nu},$

\item if $\len(\bar{\nu})>0$ then

$\hspace{.5cm}$ $(7.1)$ $\supp(p_{1})=\{\pi_{\mc(p), \bar{\gamma}}(\bar{\nu}): \bar{\gamma} \in \supp(p) \},$

$\hspace{.5cm}$ $(7.2)$ $p^{\pi_{\mc(p), 0}(\bar{\nu})} =p^{0},$

$\hspace{.5cm}$ $(7.3)$ $p^{\pi_{\mc(p), \bar{\gamma}}(\bar{\nu})} =\langle \rangle,$ if $\bar{\gamma} \neq \bar{E}_\kappa,$

$\hspace{.5cm}$ $(7.4)$ $\mc(p_{1})=\bar{\nu},$

$\hspace{.5cm}$ $(7.5)$ $T^{p_{1}}=T^{p} \upharpoonright \bar{\nu},$

\item if $\len(\bar{\nu})=0$ then

$\hspace{.5cm}$ $(8.1)$ $\supp{p_{1}} =\{ \pi_{\mc(p),0}(\bar{\nu}) \},$

$\hspace{.5cm}$ $(8.2)$ $ p_{1}^{\pi_{\mc(p),0}(\bar{\nu})}=p^{0},$

$\hspace{.5cm}$ $(8.3)$ $\mc(p_{1})=\bar{\nu}^{0},$

$\hspace{.5cm}$ $(8.4)$ $T^{p_{1}}= \emptyset.$
\end{enumerate}
\end{definition}
We use $(p_{\langle \bar{\nu} \rangle})_{0}$ and $(p_{\langle \bar{\nu} \rangle})_{1}$ for $p_{0}$ and $p_{1}$ respectively. We also let $p_{\langle \bar{\nu_{1}}, \bar{\nu_{2}} \rangle }= (p_{\langle \bar{\nu}_{1}\rangle})_{1} ^{\frown} (p_{\langle \bar{\nu}_{1} \rangle})_{0 \langle \bar{\nu_{2}} \rangle}$ and so on.

The above definition is the key step in the definition of the forcing relation $\leq.$
\begin{definition}
For $p, q \in \mathbb{Q}_{\bar{E}},$ we say $p$ is a $1-$point extension of $q$ ($p \leq^{1} q$) iff
\begin{center}
$p=p_{n+1} ^{\frown} ...^{\frown} p_{0} $

$q=q_{n} ^{\frown} ...^{\frown} q_{0} $
\end{center}
and there is $0 \leq k \leq n$ such that
\begin{itemize}
\item $\forall i < k, p_{i}, q_{i} \in \mathbb{Q}_{\bar{\mu}_{i}}^{*}, p_{i} \leq^{*} q_{i},$
\item $\exists \bar{\nu} \in T^{q_{k}}, (p_{k+1}) ^{\frown}p_{k} \leq^{*} (q_{k})_{ \langle \bar{\nu} \rangle}$
\item $\forall i > k, p_{i+1}, q_{i} \in \mathbb{Q}_{\bar{\mu}_{i}}^{*}, p_{i+1} \leq^{*} q_{i},$
\end{itemize}
where $\bar{\mu}_{0}=\bar{E}.$
\end{definition}
\begin{definition}
For $p, q \in \mathbb{Q}_{\bar{E}},$ we say $p$ is an $n-$point extension of $q$ ($p \leq^{n} q$) iff there are $p^{n}, ..., p^{0}$ such that
\begin{center}
$p=p^{n} \leq^{1} ... \leq^{1} p^{0}=q.$
\end{center}
\end{definition}
\begin{definition}
For $p, q \in \mathbb{Q}_{\bar{E}},$ we say $p$ is an extension of $q$ ($p \leq q$) iff there is some $n$ such that $p \leq^{n} q$.
\end{definition}

In the next subsection, we show the existence of a Prikry type projection from
$\mathbb{P}_{\bar{E}}$ into $\mathbb{Q}_{\bar{E}}$, and then use the results of section 2 to prove the basic properties of
$\mathbb{Q}_{\bar{E}}$.

\subsection{Projection of $\mathbb{P}_{\bar{E}}$ into $\mathbb{Q}_{\bar{E}}$}
 In this subsection we define a Prikry type projection $\Phi: \mathbb{P}_{\bar{E}} \rightarrow \mathbb{Q}_{\bar{E}}$.
 Before doing that, we define a projection
\begin{center}
$\Phi_{\bar{E}}^*: \mathbb{P}^*_{\bar{E}} \rightarrow \mathbb{Q}^*_{\bar{E}}$
\end{center}
as follows.
Suppose $p \in \mathbb{P}_{\bar{E}}^{*},$ say
\begin{center}
$p = \{ \langle \bar{\gamma}, p^{\bar{\gamma}}\rangle: \bar{\gamma} \in s \} \cup \{\langle \bar{E}_{\alpha}, T, f, F \rangle  \}.$
\end{center}
Then  set
\begin{center}
$\Phi_{\bar{E}}^*(p) = \{ \langle \bar{\gamma}, \Phi_{\bar{E}}^*(p^{\bar{\gamma}})\rangle: \bar{\gamma} \in s \} \cup \{\langle \bar{E}_{\alpha}, T \rangle  \},$
\end{center}
where
\begin{enumerate}
\item $\Phi_{\bar{E}}^*(p^0)=p^0,$
\item $\Phi_{\bar{E}}^*(p^{\bar{\gamma}})= \langle \rangle,$ if $\bar{\gamma}\neq \bar{E}_\kappa.$
\end{enumerate}
It is evident that $\Phi_{\bar{E}}^*$ is well-defined.
\begin{lemma}
$\Phi_{\bar{E}}^*$ is a projection  from $(\mathbb{P}^*_{\bar{E}}, \leq^*)$ into $(\mathbb{Q}^*_{\bar{E}}, \leq^*).$
\end{lemma}
\begin{proof}
Parts $(1)$ and $(2)$ of Definition 2.1 are evident, so let's prove part $(3)$ of the definition. Let
\begin{center}
$p = \{ \langle \bar{\gamma}, p^{\bar{\gamma}}\rangle: \bar{\gamma} \in \supp(p) \} \cup \{\langle \mc(p), T^p, f^p, F^p \rangle  \} \in \mathbb{P}^*_{\bar{E}}$
\end{center}
and
\begin{center}
$q = \{ \langle \bar{\gamma}, q^{\bar{\gamma}}\rangle: \bar{\gamma} \in \supp(q) \} \cup \{\langle \mc(q), T^q \rangle  \} \in \mathbb{Q}^*_{\bar{E}},$
\end{center}
where $q \leq^* \Phi_{\bar{E}}^*(p).$ Further we suppose that $\mc(q) >_{\bar{E}} \mc(\Phi_{\bar{E}}^*(p)).$ Note that $\supp(\Phi_{\bar{E}}^*(p))=\supp(p)$, $\mc(\Phi_{\bar{E}}^*(p))=\mc(p)$ and $T^{\Phi_{\bar{E}}^*(p)}=T^p$, so we have
\begin{enumerate}
\item $\supp(q) \supseteq \supp(p),$
\item $mc(p) \in \supp(q),$
\item $T^q \leq \pi^{-1''}_{\mc(q), \mc(p)} T^p$.
\end{enumerate}
Let $p^*$ be any condition in  $\mathbb{P}^*_{\bar{E}}$ such that:
\begin{itemize}
\item [(4)] $\supp(p^*)=\supp(q)$ and $\mc(p^*)=\mc(q),$
\item [(5)] For $\bar{\gamma} \in \supp(p),$ $(p^*)^{\bar{\gamma}}=p^{\bar{\gamma}},$
\item [(6)] $T^{p^*} \leq T^q,$
\item [(7)] $f^{p^*}=f^p \upharpoonright \{\bar{\nu} \in T^{p^*}: \len(\bar{\nu})=0   \},$
\item [(8)] $F^{p^*}=F^p \upharpoonright \{\langle \bar{\nu}_1, \bar{\nu}_2 \rangle \in (T^{p^*})^2: \len(\bar{\nu}_1)=\len(\bar{\nu}_2)=0   \}.$
\end{itemize}
It is now easily verified that $p^* \leq^* p$ and $\Phi_{\bar{E}}^*(p^*) \leq^* q.$
The lemma follows.
\end{proof}
We are now ready to define the  projection  $\Phi: \mathbb{P}_{\bar{E}} \rightarrow \mathbb{Q}_{\bar{E}}$.
Thus suppose that $p \in \mathbb{P}_{\bar{E}},$ say
\begin{center}
$p=p_{n} ^{\frown} \dots ^{\frown} p_1 ^{\frown} p_{0} $
\end{center}
where
\begin{itemize}
\item $p_{0} \in \mathbb{P}_{\bar{E}}^{*}, \kappa^{0}(p_{0}^{0}) \geq \kappa^{0}(\bar{\mu}_{1}),$
\item $p_{1} \in \mathbb{P}_{\bar{\mu}_{1}}^{*}, \kappa^{0}(p_{1}^{0}) \geq \kappa^{0}(\bar{\mu}_{2}),$

$\vdots$

\item $p_{n} \in \mathbb{P}_{\bar{\mu}_{n}}^{*}.$
\end{itemize}
and $\langle \bar{\mu}_{n}, ..., \bar{\mu}_{1}, \bar{E} \rangle$ is a $^{0}-$inceasing sequence of extender sequence systems. Set
\begin{center}
$\Phi(p)=\Phi_{\bar{\mu}_n}^*(p_{n}) ^{\frown} \dots ^{\frown} \Phi_{\bar{\mu}_1}^*(p_{n}) ^{\frown} \Phi_{\bar{E}}^*(p_{0}).$
\end{center}
It is easily seen that $\Phi(p) \in \mathbb{Q}_{\bar{E}},$ and that
$\Phi: \mathbb{P}_{\bar{E}} \rightarrow \mathbb{Q}_{\bar{E}}$ is well-defined.
\begin{lemma}
$\Phi$ is a projection  from $(\mathbb{P}_{\bar{E}}, \leq^*)$ into $(\mathbb{Q}_{\bar{E}}, \leq^*).$
\end{lemma}
\begin{proof}
By Lemma 8.1 and the definition of $\leq^*$ relation.
\end{proof}
To prove the main result, namely that $\Phi$ is a  Prikry type projection, we need the following simple observation.
\begin{lemma}
Assume $p \in \mathbb{P}_{\bar{E}}$ and  $\bar{\nu} \in T^{p}$ is such that $p_{\langle \bar{\nu}  \rangle}$ is well-defined (cf. Definition 6.5). Then
\[
\Phi(p_{\langle \bar{\nu}  \rangle})=\Phi(p)_{_{\langle \bar{\nu}  \rangle}}.
\]
\end{lemma}
\begin{proof}
Write $p_{\langle \bar{\nu}  \rangle}=p_1^{\frown} p_0,$ where $p_0, p_1$ are defined as in Definition 3.13. Then by definition of $\Phi,$ we have
\[
\Phi(p_{\langle \bar{\nu}  \rangle})=\Phi(p_1) ^{\frown} \Phi(p_0).
\]
Also let
\[
\Phi(p)=\{ \langle \bar{\gamma}, \Phi^*_{\bar{E}}(p^{\bar{\gamma}}) \rangle: \bar{\gamma} \in \supp(p)  \} \cup \{ \langle \bar{E}_{\mc(p)}, T^{p} \rangle  \}
\]
and write $\Phi(p)_{_{\langle \bar{\nu}  \rangle}}=q_1 ^{\frown} q_0$ as in Definition 3.22.
So it suffices to show that $\Phi(p_i)=q_i, i=0,1.$
\\
{\bf $\underline{\Phi(p_0)=q_0}$:} We have
\[
\Phi(p_0)=\{ \langle \bar{\gamma}, \Phi^*_{\bar{E}}(p_0^{\bar{\gamma}}) \rangle: \bar{\gamma} \in \supp(p)  \} \cup \{ \langle \bar{E}_{\mc(p)}, T^{p} \setminus \bar{\nu} \rangle  \},
\]
where $\Phi^*_{\bar{E}}(p_0^0)=p_0^0=\pi_{\mc(p),0}(\bar{\nu}),$ and $\Phi^*_{\bar{E}}(p_0^{\bar{\gamma}})=\langle \rangle,$ for all $\bar{\gamma} \neq \bar{E}_\kappa.$ It is now evident from the Definition 3.22 that $\Phi(p_0)=q_0.$
\\
{\bf $\underline{\Phi(p_1)=q_1}$:}
We consider two cases.
\begin{enumerate}
\item  $\underline{\len(\nu)> 0}$:
We have
\[
\Phi(p_1)=\{ \langle \pi_{\mc(p),\bar{\gamma}}(\bar{\nu}), \Phi^*_{\bar{E}}(p_1^{\pi_{\mc(p),\bar{\gamma}}(\bar{\nu})}) \rangle: \bar{\gamma} \in \supp(p)  \} \cup \{ \langle \bar{\nu}, T^{p} \upharpoonright \bar{\nu} \rangle  \},
\]
where $\Phi^*_{\bar{E}}(p_1^{\pi_{\mc(p),0}(\bar{\nu})})=p^0,$ and $\Phi^*_{\bar{E}}(p_1^{\pi_{\mc(p),\bar{\gamma}}(\bar{\nu})})=\langle \rangle,$ for all $\bar{\gamma} \neq \bar{E}_\kappa.$ It is now evident from the Definition 3.22 that $\Phi(p_1)=q_1.$

\item $\underline{\len(\nu)= 0}$: Then
\[
\Phi(p_1)=\{ \langle \pi_{\mc(p),0}(\bar{\nu}), \Phi^*_{\bar{E}}(p^0) \rangle  \} \cup \{ \langle \bar{\nu}^0, \emptyset \rangle  \}=
\{ \langle \pi_{\mc(p),0}(\bar{\nu}), p^0 \rangle  \} \cup \{ \langle \bar{\nu}^0, \emptyset \rangle  \},
\]
and by Definition 3.22, it is again equal to $q_1$.
\end{enumerate}
\end{proof}

\begin{theorem}
$\Phi$ is a projection of Prikry type from $(\mathbb{P}_{\bar{E}}, \leq, \leq^*)$ into $(\mathbb{Q}_{\bar{E}}, \leq, \leq^*).$
\end{theorem}

\begin{proof}
We prove the theorem in steps.
\begin{itemize}
\item We have
\[
\Phi(1_{\mathbb{P}_{\bar{E}}})= \Phi^*_{\bar{E}}(1_{\mathbb{P}^*_{\bar{E}}})= 1_{\mathbb{Q}^*_{\bar{E}}}  =1_{\mathbb{Q}_{\bar{E}}},
\]
and so $\Phi$ sends $1_{\mathbb{P}_{\bar{E}}}$ to $1_{\mathbb{Q}_{\bar{E}}}$.

\item To show $\Phi$ is order preserving, let $p \leq q$ in $\mathbb{P}_{\bar{E}}.$
So we can find $n<\omega$ with $p \leq^n q.$
If $n=0,$ then $p \leq^* q,$ and the result follows from the definition of $\leq^*$ relation and Lemma 3.26.
So let's assume $n>0.$ We consider the case $n=1,$ as it is enough general to give the main idea of the general case.
Then $p \leq^1 q,$ which means
$p = p^1$$^{\frown} p^0,$ and there exists some $\bar{\nu} \in T^{q}$ such that
$p  \leq^* q_{\langle \bar{\nu}  \rangle}$. Let $q_{\langle \bar{\nu}  \rangle}=q^1$$^{\frown} q^0.$ Also
let $\bar{\mu}$ be such that $p^1, q^1 \in \mathbb{P}_{\bar{\mu}}.$ Then

$\hspace{1.cm}$ $\Phi(p)=\Phi^*_{\bar{\mu}}(p^1)^{\frown} \Phi^*_{\bar{E}}(p^0)$

$\hspace{1.8cm}$ $ \leq^* \Phi^*_{\bar{\mu}}(q^1)^{\frown} \Phi^*_{\bar{E}}(q^0)$

$\hspace{1.8cm}$ $=\Phi(q_{\langle \bar{\nu}  \rangle})$

$\hspace{1.8cm}$ $=\Phi(q)_{_{\langle \bar{\nu}  \rangle}}$ (by Lemma 3.28)

Hence $\Phi(p) \leq \Phi(q).$

\item Now let $p\in \mathbb{P}_{\bar{E}},$ $q \in \mathbb{Q}_{\bar{E}}$ and suppose that $q \leq \Phi(p).$
Let $n$ be such that $q \leq^n \Phi(p).$ If $n=0$, then by applying Lemma 3.26, we can find $p^* \leq^* p $ with $\Phi(p^*) \leq^* q.$
Suppose $n>0.$ As above, we consider the case $n=1;$ the general case can be proved similarly. We have $q \leq^1 \Phi(p),$
so $q = q^1$$^{\frown} q^0,$ and there exists some $\bar{\nu} \in T^{\Phi(p)}=T^p$ such that
$q  \leq^* \Phi(p)_{\langle \bar{\nu}  \rangle}$

Let $p_{\langle \bar{\nu}  \rangle}=p_{\langle \bar{\nu}  \rangle}^1$$^{\frown} p_{\langle \bar{\nu}  \rangle}^0,$ and note that we have
 \begin{center}
 $\Phi(p)_{\langle \bar{\nu}  \rangle}=\Phi(p_{\langle \bar{\nu}  \rangle})=\Phi^*_{\bar{\mu}}(p_{\langle \bar{\nu}  \rangle}^1)$$^{\frown} \Phi^*_{\bar{E}}(p_{\langle \bar{\nu}  \rangle}^0).$
\end{center}
where $\bar{\mu}$ is such that $q^1, \Phi(p_{\langle \bar{\nu}  \rangle}^1) \in \mathbb{P}_{\bar{\mu}}.$

Then $q^1 \leq^* \Phi(p_{\langle \bar{\nu}  \rangle}^1)$ and $q^0 \leq^* \Phi(p_{\langle \bar{\nu}  \rangle}^0)$,  so by Lemma 3.27, we can find $(p^*)^1 \in \mathbb{P}_{\bar{\mu}}$ and $(p^*)^0 \in \mathbb{P}_{\bar{E}}$ such that
\begin{enumerate}
\item $(p^*)^1 \leq^* p_{\langle \bar{\nu}  \rangle}^1$,
\item $(p^*)^0 \leq^* p_{\langle \bar{\nu}  \rangle}^0,$
\item $\Phi((p^*)^1) \leq^* q^1,$
\item $\Phi((p^*)^0) \leq^* q^0.$
\end{enumerate}
Let $p^* =(p^*)^1$$^{\frown} (p^*)^0.$ Then $p^* \leq^* p_{\langle \bar{\nu}  \rangle},$ so in particular $p^* \leq p.$ Also we have
\[
\Phi(p^*)=\Phi^*_{\bar{\mu}}((p^*)^1) ^{\frown} \Phi^*_{\bar{E}}((p^*)^0) \leq^* q^1~^{\frown} q^0=q.
\]
The result follows
\end{itemize}

\end{proof}

\subsection{More on $\mathbb{Q}_{\bar{E}}$}
We now  return to the forcing $\mathbb{Q}_{\bar{E}}$ and prove some of its basic properties.
\begin{lemma}
$(\mathbb{Q}^*_{\bar{E}}, \leq^*)$ is $\kappa^+$-c.c.
\end{lemma}
\begin{proof}
If $p, q \in \mathbb{Q}^*_{\bar{E}}$ are such that $p^0=q^0,$ then $p$ and $q$ are compatible. Now the result follows from the fact that there are only $\kappa$-many $p^0$'s.
\end{proof}
The following is an immediate corollary of the above lemma.
\begin{lemma}
$(\mathbb{Q}_{\bar{E}}, \leq)$ is $\kappa^+$-c.c.
\end{lemma}

\begin{lemma}
$(\mathbb{Q}_{\bar{E}}, \leq, \leq^*)$ satisfies the Prikry property.
\end{lemma}
\begin{proof}
By Lemma 2.6 and Theorem 3.29.
\end{proof}
We also have the following splitting lemma, whose proof is easy.
\begin{lemma} (the splitting lemma)
Assume
\begin{center}
$p=p_{n} ^{\frown} ...^{\frown} p_{0} \in \mathbb{Q}_{\bar{E}}$
\end{center}
where
\begin{itemize}
\item $p_{0} \in \mathbb{Q}_{\bar{E}}^{*}, \kappa^{0}(p_{0}^{0}) \geq \kappa^{0}(\bar{\mu}_{1}),$
\item $p_{1} \in \mathbb{Q}_{\bar{\mu}_{1}}^{*}, \kappa^{0}(p_{1}^{0}) \geq \kappa^{0}(\bar{\mu}_{2}),$

$\vdots$

    \item $p_{n} \in \mathbb{Q}_{\bar{\mu}_{n}}^{*},$
\end{itemize}
and let $0 < m < n.$ Let
\[
p^{\leq m}=p_{n} ^{\frown} ...^{\frown} p_{m}
\]
and
\[
p^{> m}=\langle \kappa^0(\bar{\mu}_{m}) \rangle^{\frown}  p_{m+1} ^{\frown} ...^{\frown} p_{0}.
\]
Then $p^{\leq m} \in \mathbb{Q}_{\bar{\mu}_{m}}$, $p^{> m} \in \mathbb{Q}_{\bar{E}}$, and there exists
\[
\mathbb{Q}_{\bar{E}}/ p \simeq \mathbb{Q}_{\bar{\mu}_{m}} / p^{\leq m} \times \mathbb{Q}_{\bar{E}}/ p^{> m}
\]
which is a forcing isomorphism with respect to both $\leq$ and $\leq^*$ relations. Further $(\mathbb{Q}_{\bar{E}}/ p^{> m}, \leq^*)$ is
$\kappa^0(\bar{\mu}_{m+1})$-closed.
\end{lemma}
So by standard arguments, we have the following result.
\begin{lemma}
With the same notation as above, any bounded subset of $\kappa^0(\bar{\mu}_{m+1})$ in $V^{\mathbb{Q}_{\bar{E}}/ p}$ is in fact in $V^{\mathbb{Q}_{\bar{\mu}_{m}} / p^{\leq m}}$.
\end{lemma}
Suppose that $H^*$ is $\mathbb{Q}_{\bar{E}}-$generic over $V=V^*[G]$.  Let
\[
\langle  \bar{\mu}_\alpha: \alpha < \kappa       \rangle
\]
be an enumeration of $\bigcup\{ \max(p_0^0):   p \in H^*          \}$ so that
\[
\alpha < \beta < \kappa \Rightarrow \kappa^0(\bar{\mu}_\alpha) < \kappa^0(\bar{\mu}_\beta).
\]
\begin{lemma}
$V[H^*]=V[\langle  \bar{\mu}_\alpha: \alpha < \kappa       \rangle].$
\end{lemma}
\begin{proof}
Let $H_*$ consists of those conditions $p \in \mathbb{Q}_{\bar{E}}$
such that:
\begin{enumerate}
\item If some $\bar{\nu}$ appears in $p^0,$ then $\bar{\nu}=\bar{\nu}_\alpha,$ for some $\alpha<\kappa,$
\item For any $\alpha<\kappa,$ there exists an extension $q$ of $p$ such that $\bar{\nu}_\alpha$ appears in $q$.
\end{enumerate}
It is clear that $H^* \subseteq H_*,$ so by genericity of $H^*$, it suffices to show that $H_*$ is a filter. But this is clear as any two conditions
$p,q$  in
$H_*$ have extensions $p_*, q_*$ respectively with $p_*, q_* \in H_*$ and with the same $p_*^0=q_*^0$, and then $p_*, q_*$ are compatible, and their natural common extension is still in $H_*$.
\end{proof}

Also let
\begin{center}
$C_{H^*} = \{ \max\kappa(p_{0}^{0}): p \in H^* \}.$
\end{center}
It is clear that $C_{H^*}=\{ \kappa^0(\bar{\mu}_\alpha): \alpha < \kappa    \}$. Also we have the following
\begin{lemma}
$(1)$ $C_{H^*}$ is a club in $\kappa,$

$(2)$ If $\beta < \kappa$ and $A \in V[\langle  \bar{\mu}_\alpha: \alpha < \kappa       \rangle]$
is a bounded subset of $\kappa^0(\bar{\mu}_\beta),$ then $A \in V[\langle  \bar{\mu}_\alpha: \alpha  < \beta       \rangle].$

$(3)$ $V[H^*]$ is a cardinal preserving extension of $V$.

$(4)$ $V[H^*]$ satisfies the $GCH$.
\end{lemma}
\begin{proof}
$(1)$ is a standard fact, and $(2)$ follows from Lemma 3.34. Lets prove $(3)$ and $(4).$

By Lemma 3.31, forcing with $\mathbb{Q}_{\bar{E}}$ preserves all cardinals greater than $\kappa.$
Now suppose on the contrary that forcing with $\mathbb{Q}_{\bar{E}}$ collapses cardinals, and let $\lambda$
be the least cardinal which is collapsed in $V[H^*]$. Then $\lambda=\theta^+ < \kappa$ is a successor cardinal, and there exists a subset $A \subseteq \theta$
which codes a collapsing map from $\lambda$ to $\theta$. Let $\beta<\kappa$ be the least ordinal such that $\lambda < \kappa^0(\bar{\mu}_\beta)$. Then
$\beta$ is a successor ordinal, say $\beta=\bar{\beta}+1.$ By $(2)$,
\[
A \in V[\langle  \bar{\mu}_\alpha: \alpha  < \beta       \rangle]=V[\langle  \bar{\mu}_\alpha: \alpha  < \bar{\beta}       \rangle].
\]
But $V[\langle  \bar{\mu}_\alpha: \alpha  < \bar{\beta}       \rangle]$ is a generic extension of $V$ by a $(\kappa^0(\bar{\mu}_{\bar{\beta}}))^+$-c.c.
forcing notion (by Lemma 3.31) and $ \lambda \geq (\kappa^0(\bar{\mu}_{\bar{\beta}}))^+$. We get a contradiction and $(3)$ follows.

Now, we show that in the extension by $\mathbb{Q}_{\bar{E}}, 2^\kappa=\kappa^+.$ For each $s \subseteq \bar{E}$ set
\[
\mathbb{Q}_{\bar{E}} \upharpoonright s = \{ p \in \mathbb{Q}_{\bar{E}}: \supp(p_0) \cup \{\mc(p_0)   \} \subseteq s            \}.
\]

By the same arguments as in \cite{merimovich2}, claim 10.3,
\[
P^{V[H^*]}(\kappa) = \bigcup_{s \in ([\bar{E}]^{\kappa})_V} P^{V[H^* \upharpoonright s]}(\kappa),
\]
where $H^* \upharpoonright s = H^* \cap \mathbb{Q}_{\bar{E}} \upharpoonright s.$ For any such $s$, we can assume it has a $\leq_{\bar{E}}$-maximal element,
say $\bar{E}(s),$ and then using our identification of two elements (see remarks after Definition 3.19), we have in fact
\[
\mathbb{Q}_{\bar{E}} \upharpoonright s \simeq \{ p \in \mathbb{Q}_{\bar{E}}: \mc(p_0)=\bar{E}(s)    \}.
\]
\begin{claim}
$\{ p \in \mathbb{Q}_{\bar{E}}: \mc(p_0)=\bar{E}(s)    \}\simeq \{ p \in \mathbb{Q}_{\bar{E}}: \mc(p_0)=\bar{E}_\kappa    \}$
\end{claim}
\begin{proof}
For $p=p_{n} ^{\frown} ...^{\frown} p_{0} \in \mathbb{Q}_{\bar{E}}$ with $\mc(p_0)=\bar{E}(s)$ set
\[
\pi(p)=p_{n} ^{\frown} ...^{\frown} p_1 ^{\frown} p^*_{0},
\]
where $p^*_0 = \{ \langle \bar{E}_\kappa, p_0^0, \pi^{-1''}_{\bar{E}(s), \bar{E}_\kappa}(T^{p_0}) \rangle \}$.
$\pi$ defines a forcing isomorphism
from $\{ p \in \mathbb{Q}_{\bar{E}}: \mc(p_0)=\bar{E}(s)    \}$ into a dense subset of $ \{ p \in \mathbb{Q}_{\bar{E}}: \mc(p_0)=\bar{E}_\kappa    \}.$
\end{proof}
It follows that for any $s$ as above,
\[
P^{V[H^* \upharpoonright s]}(\kappa)=P^{V[H^* \upharpoonright \{\bar{E}_\kappa\}]}(\kappa).
\]
But we have $|P^{V[H^* \upharpoonright \{\bar{E}_\kappa\}]}(\kappa)|=\kappa^+$ (see  \cite{merimovich2}, claim 10.3 )
and so $|P^{V[H^*]}(\kappa)|=\kappa^+.$
Finally $(4)$ follows from the splitting Lemma 3.33, Lemma 3.34 and the above argument.
\end{proof}

\subsection{Homogeneity of the quotient forcing}
Consider the projection $\Phi: \mathbb{P}_{\bar{E}} \rightarrow \mathbb{Q}_{\bar{E}}$ given in the last subsection. So given any $H$ which is $ \mathbb{P}_{\bar{E}}$-generic
over $V$,  $H^*$, the filter generated by $\Phi[H]$, is $ \mathbb{Q}_{\bar{E}}$-generic over $V$, and we can form the quotient forcing
\[
 \mathbb{P}_{\bar{E}}/ H^* = \{ p \in  \mathbb{P}_{\bar{E}}: \Phi(p) \in H^*                 \}.
\]
We show that $ \mathbb{P}_{\bar{E}}/ H^*$ has enough homogeneity properties to guarantee that
\[
HOD^{V[H]} \subseteq V[H^*].
\]
First we prove this for $\mathbb{P}^*_{\bar{E}}$ and $\mathbb{Q}^*_{\bar{E}}$.
\begin{lemma}
Suppose $p, q \in \mathbb{P}_{\bar{E}}^*$ are such that $\Phi_{\bar{E}}^*(p)=\Phi_{\bar{E}}^*(q).$ Then there exists an isomorphism
\[
\chi_{\bar{E}}: \mathbb{P}^*_{\bar{E}}/ p \simeq \mathbb{P}^*_{\bar{E}}/ q.
\]
\end{lemma}
\begin{proof}
As $\Phi_{\bar{E}}^*(p)=\Phi_{\bar{E}}^*(q),$ we have
\begin{enumerate}
\item $\supp(p)=\supp(q)$ and $\mc(p)=\mc(q)$,
\item $p^0=q^0$,
\item $T^p=T^q,$ in particular  $\dom(f^p)=\dom(f^q)$ and $\dom(F^p)=\dom(F^q).$
\end{enumerate}
Let  $T=T^p=T^q.$
The function $R(-, -)$ is homogeneous in the following sense:
\begin{enumerate}
\item [(4)] For any $\nu \in T$ with $\len(\nu)=0,$ there exists an isomorphism
\[
\psi_{\nu}: R(\kappa(p^0), \nu^0) / f^p(\nu^0) \simeq   R(\kappa(q^0), \nu^0) / f^q(\nu^0),
\]

\item [(5)] For any $\langle \nu_1, \nu_2 \rangle \in T^2$ with $\len(\nu_1)=\len(\nu_2)=0,$ there exists an isomorphism
\[
\psi_{\nu_1, \nu_2}: R(\nu_1^0, \nu_2^0) / F^p(\nu^0_1, \nu^0_2) \simeq   R(\nu_1^0, \nu_2^0) / F^q(\nu^0_1, \nu^0_2).
\]
\end{enumerate}
Now define $\chi_{\bar{E}}$ as follows. Assume $p^* \leq^* p.$ Then $q^*=\chi_{\bar{E}}(p^*)$ is defined to be a condition  in $\mathbb{P}^*_{\bar{E}}$ such that
\begin{enumerate}
\item [(6)] $\supp(q^*)=\supp(p^*)$ and $\mc(q^*)=\mc(p^*),$

\item [(7)] $\bar{\gamma} \in \supp(p) \Rightarrow (q^*)^{\bar{\gamma}} = q^{\bar{\gamma}},$

\item [(8)] $\bar{\gamma} \in \supp(p^*)\setminus \supp(p) \Rightarrow (q^*)^{\bar{\gamma}} = (p^*)^{\bar{\gamma}},$

\item [(9)] $T^{q^*} =T^{p^*},$

\item [(10)] $\dom(f^{q^*})=\dom(f^{p^*})$ and for $\nu \in \dom(f^{q^*}),$ $f^{q^*}(\nu)=\psi_{\nu}(f^{p^*})(\nu),$

\item [(11)] $\dom(F^{q^*})=\dom(F^{p^*})$ and for $\langle \nu_1, \nu_2 \rangle \in \dom(F^{q^*}),$ $F^{q^*}(\nu_1, \nu_2)=\psi_{\nu_1, \nu_2}(F^{p^*})(\nu_1, \nu_2).$
\end{enumerate}
It is evident that $q^* \leq^* q,$ so $\chi_{\bar{E}}: \mathbb{P}^*_{\bar{E}}/ p \rightarrow \mathbb{P}^*_{\bar{E}}/ q$ is well-defined.
It is not difficult to show that $\chi_{\bar{E}}$ is in fact an isomorphism.
\end{proof}
\begin{remark}
In fact it suffices to have $\Phi_{\bar{E}}^*(p)$ and $\Phi_{\bar{E}}^*(q)$ are compatible.
\end{remark}
We now prove the main result of this subsection
\begin{theorem}
Suppose $p, q \in \mathbb{P}_{\bar{E}}$ are such that $\Phi(p)=\Phi(q).$ Then there are $p' \leq^* p, q' \leq^* q$ and an isomorphism
\[
\chi: \mathbb{P}_{\bar{E}}/ p' \simeq \mathbb{P}_{\bar{E}}/ q'.
\]
\end{theorem}
\begin{proof}
Using the factorization properties of $\mathbb{P}_{\bar{E}},$ it suffices to prove the lemma for $p, q \in \mathbb{P}^*_{\bar{E}}.$
Then $\Phi(p)=\Phi(q)$ just means $\Phi_{\bar{E}}^*(p)=\Phi_{\bar{E}}^*(q),$
and so as above
\begin{enumerate}
\item $\supp(p)=\supp(q)$ and $\mc(p)=\mc(q)$,
\item $p^0=q^0$,
\item $T^p=T^q,$ in particular  $\dom(f^p)=\dom(f^q)$ and $\dom(F^p)=\dom(F^q).$
\end{enumerate}
Let  $s=\supp(p)=\supp(q)$, $\bar{E}_\alpha=\mc(p)=\mc(q)$ and $T=T^p=T^q.$
For any $\bar{\gamma} \in s,$ let $T(\bar{\gamma}) \in \bar{E}_\alpha$ be such that for all $\bar{\nu} \in T(\bar{\gamma}),$
\begin{center}
$\bar{\nu}$ is permitted for $p^{\bar{\gamma}}$$~\Leftrightarrow~$$\bar{\nu}$ is permitted for $q^{\bar{\gamma}}$,
\end{center}
and let $T_0=\Delta^0_{\bar{\gamma} \in s}T(\bar{\gamma}) \in \bar{E}_\alpha.$
Now suppose that $T_0 \supseteq T_{1} \supseteq  \dots \supseteq T_n$ are defined such that each $T_i \in \bar{E}_\alpha,$ we define $T_{n+1}$ as follows.
Assume $\langle \bar{\nu}_0, \dots, \bar{\nu}_n \rangle \in [T_n]^n$  is $^0$-increasing. For any $\bar{\gamma} \in s$ let $T_{n, \langle \bar{\nu}_0, \dots, \bar{\nu}_n \rangle}(\bar{\gamma}) \in \bar{E}_\alpha$ be such that for any $\bar{\nu} \in T_{n, \langle \bar{\nu}_0, \dots, \bar{\nu}_n \rangle}(\bar{\gamma})$
\begin{center}
$\bar{\nu}$ is permitted for $(p^{\bar{\gamma}})_{\langle \bar{\nu}_0, \dots, \bar{\nu}_n \rangle}$$~\Leftrightarrow~$$\bar{\nu}$ is permitted for $(q^{\bar{\gamma}})_{\langle \bar{\nu}_0, \dots, \bar{\nu}_n \rangle}$,
\end{center}
and let $T_{n+1}=T_n \cap \Delta^0_{\bar{\gamma}\in s} \Delta^0_{\langle \bar{\nu}_0, \dots, \bar{\nu}_n \rangle \in [T_n]^n} T_{n, \langle \bar{\nu}_0, \dots, \bar{\nu}_n \rangle}(\bar{\gamma})$. Then $T_{n+1}\in \bar{E}_\alpha.$ Finally set $T'=\bigcap_{n<\omega}T_n.$

 Let $p', q'$ be obtained from $p,q$ by replacing $T$ by $T'$ respectively.
We define an isomorphism
\[
\chi: \mathbb{P}_{\bar{E}}/ p' \simeq \mathbb{P}_{\bar{E}}/ q'.
\]
as follows. Let $p^* \leq p.$ So we can find $n < \omega$ such that $p^* \leq^n p.$ We define $\chi(p^*)$ by induction on $n$.
\\
{\bf Case $n=0$:} Then $p^* \leq^* p,$ and set $\chi(p^*)=\chi_{\bar{E}}(p^*).$
\\
{\bf Case $n=1$:} Then $p^* \leq^1 p,$ which  means $p^* = (p^*)^1$$^{\frown} (p^*)^0,$ and there exists some $\bar{\nu} \in T^{p}$ such that
$p^* = (p^*)^1$$^{\frown} (p^*)^0 \leq^* p_{\langle \bar{\nu}  \rangle}$. Let $p_{\langle \bar{\nu}  \rangle}=p^1$$^{\frown} p^0.$ Also
let $\bar{\mu}$ be such that $p^1, (p^*)^1 \in \mathbb{P}_{\bar{\mu}}.$
Set
\[
\chi(p^*)=\chi_{\bar{\mu}}((p^*)^1)^{\frown} \chi_{\bar{E}}((p^*)^0),
\]
and note that $\chi(p^*) \leq^* \chi(p)=\chi_{\bar{\mu}}(p^1)^{\frown} \chi_{\bar{E}}(p^0).$
\\
{\bf Case $n>1$:} This case can be defined as in case $n=1.$

It is clear that $\chi: \mathbb{P}_{\bar{E}}/ p \simeq \mathbb{P}_{\bar{E}}/ q$ is well-defined. Using Lemma 3.26, $\chi$ is easily seen to be an isomorphism.
\end{proof}
The following is a consequence of above theorem, and its proof is essentially the same as in \cite{c-f-g}.
\begin{corollary}
Let $H$ be $\mathbb{P}_{\bar{E}}$-generic over $V$, and let $H^*$ be the filter generated by $\Phi[H].$ Then

$(1)$ $H^*$ is $\mathbb{Q}_{\bar{E}}$-generic over $V$,

$(2)$ $HOD^{V[H]} \subseteq V[H^*].$
\end{corollary}

\subsection{Completing the proof of Theorem 1.1.}
Finally in this subsection we complete the proof of Theorem 1.1. Let $V^*$ be the canonical core model for a $(\kappa+4)$-strong cardinal, and produce the model $V=V^*[G]$
as in subsection 3.2. Now define the corresponding forcing notions $\mathbb{P}_{\bar{E}}$ and $\mathbb{Q}_{\bar{E}}$ and let $H$ be $\mathbb{P}_{\bar{E}}$-generic over $V$.
Let $\kappa_0=\min(C^\kappa_H)$ and let $K$ be $\Col(\omega, \kappa_0^+)$-generic over $V[H].$

Also let $H^*$ be the filter generated by $\Phi[H],$ where $\Phi$ is the projection from
$\mathbb{P}_{\bar{E}}$ to $\mathbb{Q}_{\bar{E}}$ from subsection 3.6. By Corollary 3.41, $H^*$ is $\mathbb{Q}_{\bar{E}}$-generic over $V$, and we have
\[
HOD^{V[H]} \subseteq V[H^*].
\]
As $\Col(\omega, \kappa_0^+)$ is homogeneous, we have
\[
HOD^{V[H][K]}  \subseteq HOD^{V[H]} \subseteq V[H^*].
\]
But as $V^*$ is the canonical core model, we have $HOD^{V[H][K]} \supseteq V^*,$ and as
 $V[H^*]$ is a cardinal preserving extension of $V^*,$ we can conclude that $V^* \subseteq HOD^{V[H][K]} \subseteq V[H^*]$ have the same cardinals.
 But $V[H^*]$ satisfies the $GCH$ below $\kappa,$ so $HOD^{V[H][K]} \models$``$\forall \lambda< \kappa, ~ 2^\lambda=\lambda^+$''.

Now take $W=V[H][K]_\kappa.$ By Theorem 3.17, $W$ is a model of $ZFC$.
It is evident that
\[
V^*_\kappa \subseteq HOD^{W_\kappa} \subseteq (HOD^W)_\kappa,
\]
and hence as $V^*_\kappa  \subseteq (HOD^W)_\kappa$ both satisfy the $GCH$ and
have the same cardinals, we can conclude that
$HOD^{W_\kappa} \models$`` $GCH$''.
The theorem follows.

\section{$GCH$ can fail everywhere in $HOD$}
In this section we give a proof of Theorem 1.3. As we stated in the introduction, by a result of Roguski \cite{roguski}, every model $V$ of $ZFC$ has a class
generic extension $V[G]$ such that $V$ is equal to $HOD^{V[G]}$ (see also \cite{f-h-r} where some generalizations of this result are proved). The model $V[G]$
constructed in both  \cite{roguski} and \cite{f-h-r} fails to satisfy the $GCH$. We modify the above constructions, and prove the following theorem, from which
theorem 1.3 will follow easily.
\begin{theorem}
Assume $V$ is a model of $ZFC$. Then $V$ has a class generic extension $V[G]$ such that

$(1)$ $V[G] \models$``$ZFC+GCH$'',

$(2)$ $HOD$ of $V[G]$ equals $V$.
\end{theorem}
\begin{proof} We follow the proof of \cite{roguski}, and modify it using some ideas from \cite{brooke}, to make sure that
our final extension satisfies the $GCH$.
We usually work in an expanded language of $ZFC$, where some unary predicates $U_1, \dots, U_n$ are added. Then by $ZFC(U_1, \dots, U_n)$
we mean $ZFC+$all instances of the replacement for formulas of the language $\mathfrak{L}_{\in}(U_1, \dots, U_n)$. Then by
$HOD(U_1, \dots U_n)$ we denote the class of all sets that are hereditarily
$\mathfrak{L}_{\in}(U_1, \dots, U_n)$- definable with only ordinal parameters. The following is proved in \cite{roguski}
\begin{lemma}
$(a)$ $HOD(U_1, \dots U_n)$ is an inner model of $ZFC(U_1, \dots, U_n),$

$(b)$ There exists an $\mathfrak{L}_{\in}(U_1, \dots, U_n)$- definable well-ordering of the class $HOD(U_1, \dots U_n)$.
\end{lemma}
We also need the following known result (see \cite{roguski}, Lemma $C$).
\begin{lemma}
Assume $\langle V, \in, X_1, \dots, X_n \rangle$ is a model of $ZFC(U_1, \dots, U_n)$, and suppose $\PP$ is a weakly homogeneous tame
forcing notion which is definable with ordinal parameters. Then for any $\PP$-generic filter $G$ over  $\langle V, \in, X_1, \dots, X_n \rangle$,
\[
HOD(V, \in, X_1, \dots, X_n)=HOD(V[G], \in, X_1, \dots, X_n, V).
\]
\end{lemma}
We now define the combinatorial principle $\diamondsuit^*_\lambda$, which will be used as our coding oracle, that we replace it by the continuum coding function used in the proofs of \cite{roguski} and
\cite{f-h-r}.
\begin{definition}
Let $\lambda$ be a regular cardinal. A $\diamondsuit^*_\lambda$-sequence is a sequence $\bar{D}=\langle D_\alpha: \alpha < \lambda  \rangle$ such that
\begin{enumerate}
\item $\forall \alpha < \lambda, ~ D_\alpha \subseteq P(\alpha)$,
\item $\forall \alpha < \lambda, ~ |D_\alpha| \leq \aleph_0+|\alpha|,$
\item  For every $X \subseteq \lambda, \{\alpha < \lambda: X \cap \alpha \in D_\alpha  \}$ contains a closed unbounded (club) subset of $\lambda.$
\end{enumerate}
$\diamondsuit^*_\lambda$ holds if a $\diamondsuit^*_\lambda$-sequence exists.
\end{definition}
 A proof of the following lemma can be found in \cite{brooke}.
\begin{lemma}
Assume $GCH$ holds and $\lambda$ is a successor cardinal.

$(a)$ There exists a weakly homogeneous $\lambda$-closed $\lambda^+$-c.c. forcing notion $Add(\diamondsuit^*_\lambda)$ of size $\lambda^+$ which  forces ``$\diamondsuit^*_\lambda$''.

$(b)$ Forcing with $Add(\lambda, \lambda^+)$ forces $\neg\diamondsuit^*_\lambda.$
\end{lemma}
Now let $V$ be a model of $ZFC$. We define the required generic extension $V[G]$ in three steps:

{\bf Step 1.} Let $\PP^1$ be the forcing for adding a global well-ordering of the universe. A condition in $\PP^1$
is of the form $p=\langle  \alpha_p, <_{p} \rangle,$ where $\alpha_p$ is an ordinal and $<_p$ is a well-ordering of $V_{\alpha_p}.$
For $p, q\in \PP^1,$ we say $p \leq q$ iff $\alpha_p \geq \alpha_q$ and $<_q = <_p \cap (V_{\alpha_q}\times V_{\alpha_q})$.
Clearly $\PP^1$ is set-closed, so forcing with it does not add any new sets. Let $G^1$ be $\PP^1$-generic over $V$. Then using Lemma 4.3,
\[
\langle  V, \in, G^1       \rangle\models \text{``~}ZFC(U)+V=HOD(U)\text{''}.
\]

{\bf Step 2.}
Force  by $\PP^2,$ the canonical forcing  for $GCH$. It
is defined as the reverse Easton iteration of forcings
\begin{center}
$\PP^2=\langle \langle \PP^2_\gamma: \gamma\in On   \rangle, \langle \lusim{\mathbb{Q}}^2_\gamma: \gamma\in On   \rangle \rangle$
\end{center}
where at each step $\gamma,$ if $\gamma$ is a cardinal in $V^{\PP^2_{\gamma}},$ then $V^{\PP^2_{\gamma}}\models$``$\mathbb{Q}^2_\gamma=\Add(\gamma^+, 1)$'', and $V^{\PP^2_{\gamma}}\models$``$\mathbb{Q}^2_\gamma$ is the trivial forcing notion'' otherwise. The following lemma is known.
\begin{lemma}
Let $G^2$ be $\PP^2$-generic over $\langle V, \in ,G^1 \rangle$. Then

$(a)$ $\langle V[G^2], \in, V, G^1 \rangle \models$``$ZFC(U_1, U_2)+GCH+$the global axiom of choice'',

$(b)$ The forcing $\PP^1$ is weakly homogeneous, in particular $HOD(V[G^2], \in,  G^1, V)=V.$
\end{lemma}
It also follows that in $\langle V[G^2], \in,  G^1, V \rangle $ there exists a class $K$ of ordinals which is $\mathfrak{L}_{\in}(U_1,U_2)$-definable such that $V=L[K]$. So
\[
V=L[K]=HOD(V[G^2], \in, G^1, V).
\]

{\bf Step 3.}
Now force over $V[G^2]$ by the forcing notion $\PP^3$, which is defined as the Easton support product of forcing notions $\PP^3_\alpha,$ where for each ordinal $\alpha,$
\begin{center}
 $\PP^3_\alpha=$ $\left\{
\begin{array}{l}
        Add(\diamondsuit^*_{\aleph_{\alpha+1}}) \hspace{2.4cm} \text{ if } \alpha \in K,\\
        Add(\aleph_{\alpha+1}, \aleph_{\alpha+2}) \hspace{1.65cm} \text{ if } \alpha \notin K.
     \end{array} \right.$
\end{center}
Let $G^3$ be $\PP^3$-generic over $\langle V[G^2], \in, G^1, V \rangle $, and let $W=V[G^2][G^3].$
In $W$, the class $K$ is $\mathfrak{L}_{\in}$-definable, and so
\[
V= L[K] \subseteq HOD^{W}.
\]
On the other hand, as the forcing $\PP^3$ is also weakly homogeneous, using Lemma 4.3, we have
\[
HOD^{W} \subseteq HOD(V[G^2][G^3], \in, G^1, V, V[G^2]) = HOD(V[G^2], \in, G^1, V)= V.
\]
It follows that $HOD^W=V,$ and the theorem follows.
\end{proof}
We are now ready to complete the proof of Theorem 1.3.

\emph{Proof of Theorem 1.3.} Suppose that $V\models$``$ZFC+GCH+ \kappa$ is a $(\kappa+4)$-strong cardinal''. By the results of section 3, there exists a generic extension $V[G_1]$ of $V$ such that
\begin{enumerate}
\item $\kappa$ remains an inaccessible cardinal in $V[G_1],$
\item $V[G_1] \models$``$\forall \lambda, 2^\lambda=\lambda^{+3}$''.
\end{enumerate}
The proof of Theorem 4.1 gives a generic extension $V[G_1][G_2]$ of $V[G_1]$ such that
\begin{enumerate}
\item [(3)] $\kappa$ remains an inaccessible cardinal in $V[G_1][G_2],$
\item [(4)] $V[G_1][G_2] \models$``$GCH$'',
\item [(5)] $V[G_1]_\kappa=HOD^{V[G_1][G_2]_{\kappa}}.$
\end{enumerate}
So it suffices to take $W=V[G_1][G_2].$ \hfill$\Box$

\section{Some generalizations and open problems}
In this section we consider some possible generalizations of our results and pose some questions. As  the proof of Theorem 1.1 shows, the models $W$
and $HOD^W$ (and hence $W_\kappa$ and $HOD^{W_\kappa}$) have different cardinals. So it is natural to ask if these models can have the same cardinals.
The next theorem gives a positive answer to this question.
\begin{theorem}
Assume $V \models$``$ZFC+GCH+$there exists a $(\kappa+4)-$strong cardinal $\kappa$''. Then there is a cardinal preserving generic extension $W$ of $V$ such that:

$(1)$ $\kappa$ remains inaccessible in $W$,

$(2)$ $V_\kappa^W=W_\kappa \models$``$~ZFC+ \forall \lambda, 2^{\lambda} > \lambda^+$''.

$(3)$ $HOD^{W_\kappa} \models$``$GCH$'',

$(4)$ The models $W_\kappa$ and $HOD^{W_\kappa}$ have the same cardinals.
\end{theorem}
\begin{proof}
The proof is similar to the proof of Theorem 1.1, with few changes.   In our preparation model, we remove the guiding forcing notion $\MRB_U^{\Col}$ from the definition
of $\MRB_U.$ Then the forcing notion $\PP_{\bar{E}}$ is defined as before using a guiding generic for this new forcing notion.  The resulting model $V[H]$, where $H$ is $\PP_{\bar{E}}$-generic over $V$, satisfies the following properties:
\begin{enumerate}
\item $V$ and $V[H]$ have the same cardinals,
\item $\kappa$ remains an inaccessible cardinal in  $V[H],$
\item $C^\kappa_H$ is a club of $\kappa,$
\item If $\alpha < \alpha^*$ are two successive points in $C^\kappa_H,$ then $2^{\alpha^{+}}=\alpha^{+4}, 2^{\alpha^{++}}=\alpha^{+5}, 2^{\alpha^{+3}}=\alpha^{+6},$
$2^{\alpha^{+4}}=(\alpha^*)^{+}, 2^{\alpha^{+5}}=(\alpha^*)^{++}$ and $2^{\alpha^{+6}}=(\alpha^*)^{+3}$.
\end{enumerate}
Now force over $V[H]$ using $Add(\omega, \lambda^{+3}),$ where $\lambda=\min(C^\kappa_H)$, and let $K$ be the resulting generic extension. It is clear that $VH][K]$
is a cardinal preserving extension of $V$, and
\[
V[H][K]\models\text{~``}\kappa \text{~is an inaccessible cardinal~}+\forall \lambda<\kappa, 2^\lambda > \lambda^+ \text{~''}.
\]
The forcing notion $\mathbb{Q}_{\bar{E}}$ is defined as before, and the rest of the arguments from section 3 work without any change.
\end{proof}
We do not know the answer to the following question.
\begin{question}
Can there be a model $V$ of $ZFC$ such that
\begin{enumerate}
\item $V \models$``$~\forall \kappa, 2^\kappa=\kappa^{+3}$'',
\item $HOD^V \models$``$GCH$'',
\item $V$ and $HOD^V$ have the same cofinalities.
\end{enumerate}
\end{question}
%On the other hand, it is clear that the models $W$ and $HOD^W$ of Theorem 1.3 can not have the same cardinals.
%But the following question seems natural.
%\begin{question}
%Does any model $V$ of $ZFC$ have a cofinality preserving class generic extension $V[G]$ such that $HOD^{V[G]}=V$?
%\end{question}

%\begin{question}
%Is there a model $V$ of $ZFC$ such that
%\begin{enumerate}
%\item $V \models$``$~\forall \kappa, 2^\kappa > \kappa^+$'',
%\item $HOD^V \models$``$GCH$''
%\item $V=HOD[R]$, for some real $R$.
%\end{enumerate}
%\end{question}
%The following question is asked in \cite{Zadrozny}: Assume $0^\sharp$ does not exists. Let $M$ be a model of $ZFC$. Is there a model $N$ of $ZFC$ extending $M$
%with $ON^N=ON^M$ such that $HOD^N=L?$ Our next result gives a negative answer to this question in a strong way.
%\begin{theorem}
%Any model $V$ of $ZFC$ has a class generic extension $V'\models$``$ZFC$'' such that if $W \supseteq V'$ is a model of $ZFC$ with
%$ON^W=ON^{V'},$ then $HOD^W \neq V.$
%\end{theorem}
%\begin{proof}
%By Jensen's coding theorem, $V$ has a class generic extension $V[G]\models$``$ZFC$'' such that for some real $R \in V[G],~ R$ is not set generic over $V$
%and $V[G]=L[R]$. We show that $V'=V[G]$
%is as required. Thus suppose $W \supseteq V'$ is a model of $ZFC$, and suppose on the contrary that $HOD^W=V.$ Then as $R\in W,$
%by a result of Vopenka, $R$ is set generic over $HOD^W=V,$ and this is a contradiction.
%\end{proof}

\subsection*{Acknowledgements}
The author thanks Sy Friedman for his nice question and his interest in the results of this paper. 
He also thanks the referee of the paper for many helpful comments and corrections.

School of Mathematics, Institute for Research in Fundamental Sciences (IPM), P.O. Box:
19395-5746, Tehran-Iran.

E-mail address: golshani.m@gmail.com

\end{document}